\documentclass[12pt]{amsart}
\usepackage{graphicx}
\usepackage{graphicx}
\usepackage{epsfig}
\usepackage{pstricks}

\newtheorem{thm}{Theorem}[section]

\newtheorem{defn}[thm]{Definition}

\newtheorem{example}[thm]{Example}

\usepackage{amsmath}
\usepackage{amsxtra}
\usepackage{amscd}
\usepackage{amsthm}
\usepackage{amsfonts}
\usepackage{amssymb}
\usepackage{eucal}
\newcommand{\bmb}{\left( \begin{array}{rr}}
\newcommand{\enm}{\end{array}\right)}

\newcommand{\Z}{{\mathbb Z}}

\newcommand{\bt}{{\mathbf t}}

\newcommand{\bk}{{\mathbf k}}

\newcommand{\al}{{\alpha}}

\numberwithin{equation}{section}

\begin{document}

\title{Bessenrodt-Stanley polynomials and the octahedron recurrence}
\author{Philippe Di Francesco} 
\address{
Department of Mathematics, University of Illinois at Urbana-Champaign, MC-382, 
1409 W\ Green St., Urbana, IL 61801, U.S.A.
e-mail: philippe@illinois.edu
}

\begin{abstract}
We show that a family of multivariate polynomials recently introduced by Bessenrodt and Stanley can be expressed as solution
of the octahedron recurrence with suitable initial data. This leads to generalizations and explicit expressions as path or dimer partition functions.
\end{abstract}

\date{\today}
\maketitle
\tableofcontents

\section{Introduction}

In a recent publication, Bessenrodt and Stanley \cite{BS} introduced a family of multivariate polynomials attached to any
partition $\lambda$, generalizing a construction by Berlekamp \cite{BKP}. These were defined as weighted sums over sub-diagrams
of the Young diagram of $\lambda$.

In this paper, we show that these polynomials
may be viewed as the restriction of particular solutions of the octahedron recurrence relation in a three-dimensional half-space. The octahedron
recurrence is a special case of so-called $T$-system, introduced in the context of generalized Heisenberg integrable quantum 
spin chains with Lie group symmetries \cite{KNS,KNS11}.
The octahedron recurrence has received much attention over the last decade for its combinatorial interpretation in terms of 
domino tilings of the  Aztec diamond \cite{CEP} and generalizations thereof \cite{SPY,Henriques,DF,DFK12,DF13,DSG14}. 
More generally, the $A$ type $T$-systems possess the positive Laurent property: their solutions may be expressed  
as Laurent polynomials of any admissible initial data, with non-negative integer coefficients. 
In Ref.\cite{DFK08}, this property of the $T$-systems was connected to an underlying cluster 
algebra structure \cite{FZI}, and further confirmed by providing an explicit solution based on some representation using a flat
two-dimensional connection \cite{DF}, leading to expressions as partition function of weighted paths on oriented graphs or networks,
or equivalently of weighted dimer coverings of suitable bipartite planar graphs.

Our punchline is the following: the polynomials of Bessenrodt and Stanley were shown to obey particular determinantal identities \cite{BS},
which we interpret as initial conditions for the half-space octahedron recurrence, for which the partition $\lambda$ determines the geometry of the initial
data. Reversing the logic, and fixing the values of 
these determinants, we may express the solutions of the octahedron recurrence
as Laurent polynomials thereof, as path or dimer partition functions from \cite{DF,DF13}. 
This produces a general family of multivariate Laurent polynomials attached to any partition, that reduce to the polynomials of \cite{BS}, for special
choices of the initial data. 

The paper is organized as follows. 
In Section 2 we describe our new family of multivariate Laurent polynomials attached to a partition $\lambda$, and how they reduce
to the polynomials of \cite{BS}. 
Section 3 is devoted to 
a survey of the $T$-system and its general solutions in terms of network partition functions. We describe in particular the structure of admissible initial data, which take the form of initial value assignments along a ``stepped surface", and show how this data encodes an oriented weighted graph or network.
In Section 4, we show that the situation of \cite{BS} corresponds to choosing a particular ``steepest" initial data stepped surface, that forms a kind of fixed slope roof above the Young diagram of $\lambda$. The corresponding network is particularly simple, as it takes the shape of the Young diagram itself. Using the explicit solutions of the $T$-system, we write the solutions first as network partition functions (Theorem \ref{thpath},
Sect.4.3) and then as
dimer partition functions (Theorem \ref{thdim}, Sect.4.4).
Section 5 is devoted to a 3D generalization of the multivariate polynomials, by simply expressing the other solutions of the $T$-system ``under the roof". This leads us to a 3D object we call the pyramid of $\lambda$, to the boxes of which we associate other Laurent polynomials.
The latter reduce to polynomials when we apply the previous restriction of initial data. These are first expressed as  partition functions for families of non-intersecting paths on the same networks as before (Sect.5.3), and then shown to restrict to sums over nested 
sub-partitions of $\lambda$ with the same weights as before (Theorem \ref{genBS}, Sect.5.4). 
We gather a few concluding remarks in Section 6, where we present a different generalization for other boundary conditions of the octahedron recurrence.

\vskip1.cm

\noindent{\bf Acknowledgments}. We would like to thank R. Stanley for an illuminating seminar during the conference
``Enumerative Combinatorics" at the Mathematisches  ForschungsInstitut Oberwolfach in March 2014 and A. Sportiello for his great 
help in the early stage of this work. This work is supported by the NSF grant DMS 13-01636 and the 
Morris and Gertrude Fine endowment.

\section{A family of Laurent polynomials associated to Young diagrams}

We consider partitions/Young diagrams of the form $\lambda=(\lambda_1,...,\lambda_N)$ with $\lambda_i$ boxes in row $i$, and
$\lambda_1\geq \lambda_2\geq \cdots\geq \lambda_N\geq 1$. We represent the Young diagram $\lambda$ with rows
$1,2,...,N$ from top to bottom and justified on the left. The boxes of the diagram $\lambda$ are labeled by their coordinates
$(i,j)$, where $i\in [1,N]$ is the row number and $j\in [1,\lambda_i]$ the horizontal coordinate within the $i$-th row.
We write $(i,j)\in \lambda$ when the box $(i,j)$ belongs to $\lambda$. For later use, for any $(a,b)\in \lambda$, we 
define the sub-diagram $\lambda_{a,b}\subset \lambda$ obtained by erasing all the boxes of $\lambda$ that are above
the row $a$ and to the left of the column $b$ (it is the intersection of the South East corner from $(a,b)$ with $\lambda$).

We also consider the extended Young diagram $\lambda^*$ associated to $\lambda$, obtained by adjoining a border strip
from the end of the first row to the end of the first column of $\lambda$, namely with
$\lambda_1^*=\lambda_1+1$,
$\lambda_i^*=\lambda_{i-1}+1 $, $i=2,...,N+1$.
For instance the extended diagram of $\lambda=(3,2)$ is $\lambda^*=(4,4,3)$. Note that $\lambda$ can be recovered from $\lambda^*$
by removing its first row and column.
We define the squares $S_{a,b}$ of $\lambda^*$ to be the square arrays of the form 
$S_{a,b}=\{ (i,j), \, a\leq i\leq n_{a,b}+a-1, \, b\leq j\leq n_{a,b}+b-1\}$, and such that $S_{a,b}\subset \lambda^*$
and $n_{a,b}\geq 1$ is maximal. 
The
box $(a,b)$ is called the North West (NW) corner of the square $S_{a,b}$. In other words, the square $S_{a,b}$ is the largest
square array of boxes with NW corner $(a,b)$ that fits in $\lambda^*$. The integer $n_{a,b}$ is called the size of $S_{a,b}$.
In particular, each box $(u,v)$ of the border strip
$\lambda^*\setminus \lambda$ is itself a square of size $1$, $S_{u,v}=\{(u,v)\}$.

\begin{defn}\label{lopdef}
We fix arbitrary non-zero parameters $\{\theta_{i,j}\}_{(i,j)\in \lambda^*}$ attached to the boxes of $\lambda^*$.
We now associate to each box $(i,j)\in \lambda^*$ a function $p_{i,j}\equiv p_{i,j}(\theta_\cdot)$ of all the parameters $\theta_{i,j}$
defined by  the identities: 
\begin{equation}\label{condet} \det\left(\left( p_{a+i-1,b+j-1}(\theta_\cdot) \right)_{1\leq i,j \leq n_{a,b}}\right) 
=\theta_{a,b}\qquad {\rm for}\, {\rm all}\, (a,b)\in \lambda^*
\end{equation}
namely we fix the value of the determinant of the array of functions $p_{i,j}(\theta_\cdot)$ on each square $S_{a,b}$ to be the parameter $\theta_{a,b}$.
\end{defn}

In particular, we have $p_{u,v}(\theta_\cdot)=\theta_{u,v}$ for all $(u,v)\in \lambda^*\setminus\lambda$.  That \eqref{condet} determines
the $p_{i,j}$'s uniquely will be a consequence of the rephrasing of the problem as that of finding the solution 
of the $T$-system or octahedron recurrence, subject to some particular initial condition. As a consequence of the positive 
Laurent property of the $T$-system,  we have the following main result:

\begin{thm}\label{main}
For all $(a,b)\in \lambda$, the function $p_{a,b}(\theta_\cdot)$ is a Laurent polynomial of the $\theta_{i,j}$'s with non-negative integer coefficients.
\end{thm}

\begin{example}\label{firstex}
Let us consider the Young diagram $\lambda=(2,1)$, with the following $\theta$ variables (one per box of $\lambda^*=(3,3,2)$):
$$ \begin{matrix} \theta_{1,1}=a & \theta_{1,2}=b & \theta_{1,3}=c \\
\theta_{2,1}=d & \theta_{2,2}=e & \theta_{2,3}=f \\
\theta_{3,1}=g & \theta_{3,2}=h & & \end{matrix} $$
The polynomials $p_{a,b}(\theta_\cdot)$ are:
$$\begin{matrix} 
p_{1,1}=\frac{a f h+(b+c e)(d+e g)}{e f h} 
& p_{1,2}=\frac{b+c e}{f} & p_{1,3}=c \\
p_{2,1}=\frac{d+e g}{h} & p_{2,2}=e & p_{2,3}=f \\
p_{3,1}=g & p_{3,2}=h & & \end{matrix} $$

\end{example}

The Laurent polynomials $p_{a,b}(\theta_\cdot)$ attached to the Young diagram $\lambda$ reduce to the polynomials introduced by 
Bessenrodt and Stanley in \cite{BS} when the variables $\theta_{i,j}$ are restricted as follows. Let $x_{i,j}$, $(i,j)\in \lambda$ be new variables
attached to the boxes of $\lambda$. 

\begin{thm}\label{restrictoBS}
Under the restrictions:
\begin{eqnarray}
\theta_{i,j}&=&1 \qquad {\rm for}\, {\rm all}\, (i,j)\in \lambda^*\setminus \lambda \ ,\nonumber \\
\theta_{a,b}&=& \prod_{r=0}^{n_{a,b}-1} \prod_{(i,j)\in \lambda_{a+r,b+r}} x_{i,j} \quad {\rm for}\, {\rm all}\, (a,b)\in \lambda\, , \label{bsrestrict}
\end{eqnarray}
the Laurent polynomials $p_{a,b}(\theta_\cdot)$ of Def.\ref{lopdef} reduce to the polynomials $p_{a,b}(x_\cdot)$ of \cite{BS}. 
\end{thm}
In particular, the change of variables cancels all denominators and produces a {\it polynomial} of the variables $x_{i,j}$.

In Section \ref{applione}, we shall construct each polynomial $p_{a,b}(\theta_\cdot)$ explicitly, as the partition function for paths on a weighted
oriented graph (network) ${\mathcal N}_{a,b}$ associated to $\lambda_{a,b}$, and alternatively as the partition function of the dimer model on a suitable bipartite graph ${\mathcal G}_{a,b}$. We give two independent proofs of Theorem \ref{restrictoBS} in Sect.\ref{sectobs}. 

\section{$A_{\infty/2}$ $T$-system and its solutions}\label{tsysect}

\subsection{$T$-system and initial data}

In the case of $A$ type
Lie algebras, the $T$-system takes the form (also known as the octahedron recurrence):
\begin{equation}\label{octa}
T_{i,j,k+1}T_{i,j,k-1}=T_{i+1,j,k}T_{i-1,j,k}+T_{i,j+1,k}T_{i,j-1,k} 
\end{equation}
where the indices of the indeterminates $T_{i,j,k}$ are restricted to be vertices of the Face-Centered Cubic (FCC) lattice 
$L_{FCC}=\{(i,j,k)\in \Z^3,\, i+j+k=1\, {\rm  mod}\,  2\}$. This may be viewed as a 2$+$1-dimensional evolution in the discrete
time variable $k$, while $i,j$ refer to space indices.

The $A_r$ condition consists in further restricting 
$i\in [1,r]$, and to impose the additional 
boundary conditions
\begin{equation}\label{affcond} T_{0,j,k}=T_{r+1,j,k}=1 \qquad (j,k\in \Z)\end{equation}
In other words, the $A_r$ $T$-system solutions are those of the octahedron equation in-between two parallel planes $i=0$ and $i=r+1$,
that take boundary value $1$ along the two planes $i=0$ and $i=r+1$.

In the following, we will concentrate on the so-called $A_{\infty/2}$ $T$-system, where we only keep the restriction $i\geq 1$
and the boundary condition 
\begin{equation}\label{plus} T_{0,j,k}=1\end{equation}
In turn, the solutions of the $A_{\infty/2}$ $T$-system are those of the octahedron equation in the half-space $i\geq 0$, that take 
boundary values $1$ along the plane $i=0$.

The system (\ref{octa},\ref{plus}) must be supplemented by some admissible initial data $(\bk,\bt)$, consisting of:
\begin{itemize}
\item{}
(1) a ``stepped surface"
$\bk=\{ (i,j,k_{i,j})\}_{i\in\Z_+;j\in\Z}$, such that $|k_{i+1,j}-k_{i,j}|=|k_{i,j+1}-k_{i,j}|=1$ for all $i,j$; 

\item{}
(2) the following assignments of initial values $\bt=\{t_{i,j}\}_{i\in \Z_{>0};j\in \Z}$:
\begin{equation}T_{i,j,k_{i,j}}=t_{i,j} \qquad (i\in \Z_{>0};j\in Z) \end{equation}
\end{itemize}

\subsection{Solution}

In Ref. \cite{DF} the $A_r$ $T$-system was solved for arbitrary $r$ and arbitrary initial data. Note that for any finite $(i,j,k)$
above the initial data stepped surface (i.e. $k\geq k_{i,j}$), the solution $T_{i,j,k}$ only depends on finitely many 
initial values $t_{i,j}$, hence for $r$ large enough  the solution is independent of $r$, so that the solution
for the case $A_{\infty/2}$ is trivially obtained from that of $A_r$. We now describe this solution.

The solution proceeds in two steps. First, one eliminates all the variables $T_{i,j,k}$ for $i\geq 2$ in terms of the
values $T_{j,k}:=T_{1,j,k}$ by noticing that the relation \eqref{octa} is nothing but the Desnanot-Jacobi (aka Dodgson condensation)
formula for the following Hankel or discrete Wronskian determinants:
\begin{equation}
W_{i,j,k}=\det\left(\left(T_{j+b-a,k+i+1-a-b}\right)_{1\leq a,b\leq i}\right)
\end{equation}
which, together with the initial condition $W_{0,j,k}=1$ allows to identify 
\begin{equation}\label{wrons} T_{i,j,k}=W_{i,j,k}
\end{equation}
for all $i\geq 2$.

The second step consists of writing a compact form for $T_{j,k}$ for all $k\geq k_{1,j}$, by use of a formulation of the relation
\eqref{octa} as the flatness condition for a suitable $GL_2$ connection. The solution is best expressed in the following manner.
First we associate to the initial data stepped surface $\bk$ a bi-colored (grey/white) triangulation with the vertices of $\bk$,
defined via the following local rules, where we indicate the value of $k_{i,j}$ at each vertex 
of the projection onto the $i,j$ plane:
\begin{equation}\label{eqrules}
\raisebox{-1.1cm}{\hbox{\epsfxsize=14.cm \epsfbox{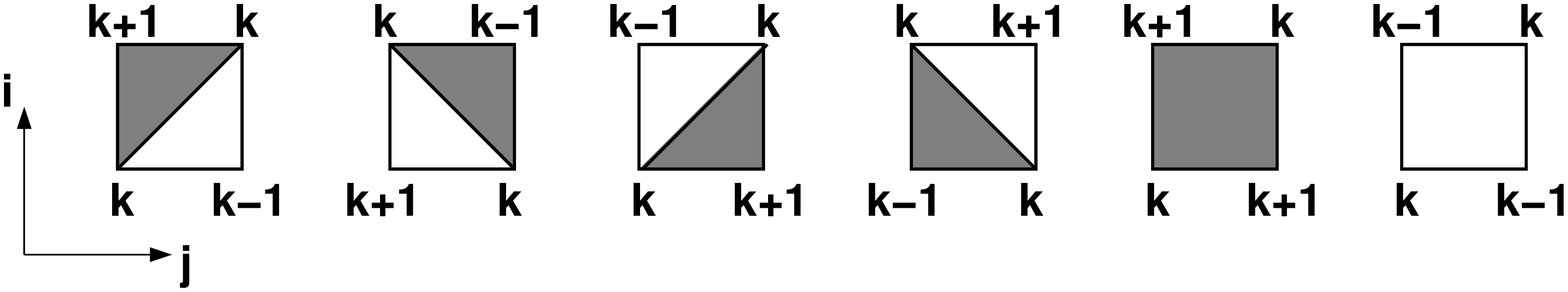}}} 
\end{equation}
The last two cases give rise to two choices  of triangulations each, due to the tetrahedron ambiguity (there are two ways of defining
a pair of adjacent triangles with the vertices of a regular tetrahedron, namely the two choices of diagonals of the white/grey square
in projection), but our construction is independent of these choices.
We may now decompose the stepped surface into lozenges made of a grey and a white triangle sharing an edge perpendicular to
the $k$ axis, and we supplement the single triangles of the bottom layer $i=1,2$ with a lower triangle of 
opposite color with bottom vertex in the $i=0$ plane. 

\begin{example}\label{flatex}
Let us consider the ``flat" initial data surface $\bk_0$ with $k_{i,j}=i+j +1\, {\rm mod} \, 2\in \{0,1\}$. Picking a particular choice
of diagonal in the tetrahedron ambiguities, we may decompose the surface as follows:
$$ \raisebox{1.1cm}{\hbox{\epsfxsize=7.cm \epsfbox{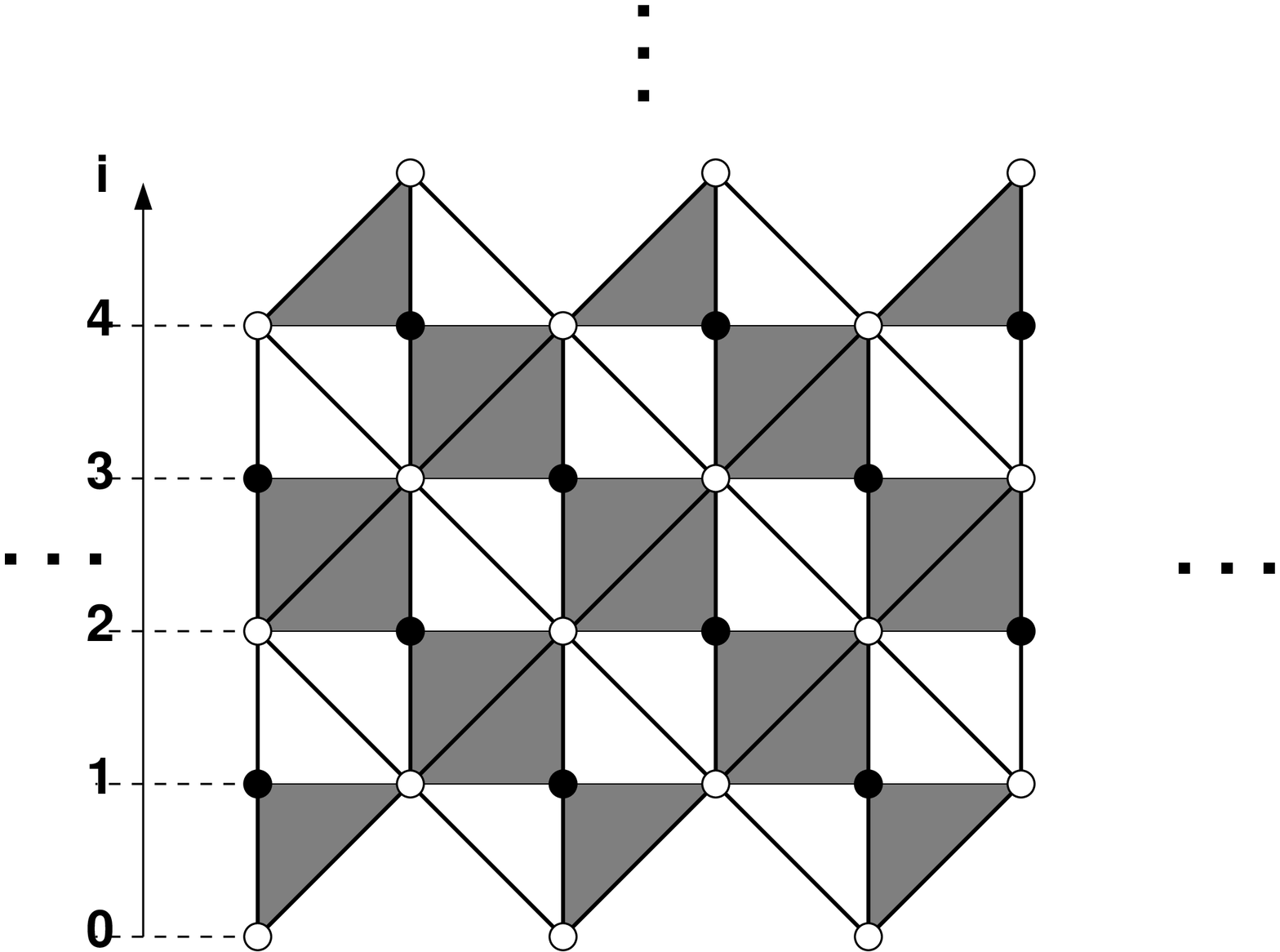}}}    $$
where we have represented with solid (resp empty) dots the vertices with $k=1$ (resp. $k=0$).
\end{example}

To each lozenge, we associate a $2\times 2$ matrix according to the following rule:
\begin{equation}
\raisebox{-1.2cm}{\hbox{\epsfxsize=2.cm \epsfbox{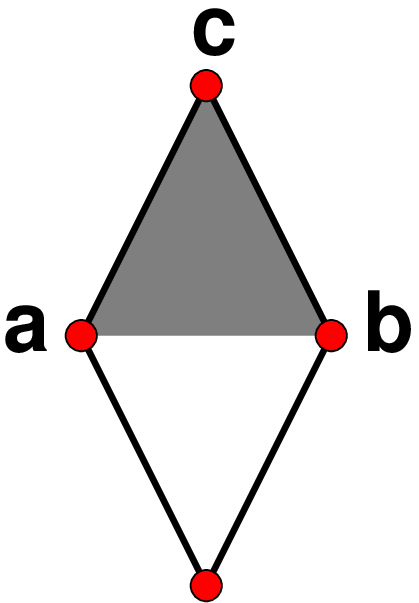}}}\to  U(a,b,c)=\begin{pmatrix} 1 & 0 \\ \frac{c}{b} & \frac{a}{b} \end{pmatrix}\qquad 
\raisebox{-1.4cm}{\hbox{\epsfxsize=2.cm \epsfbox{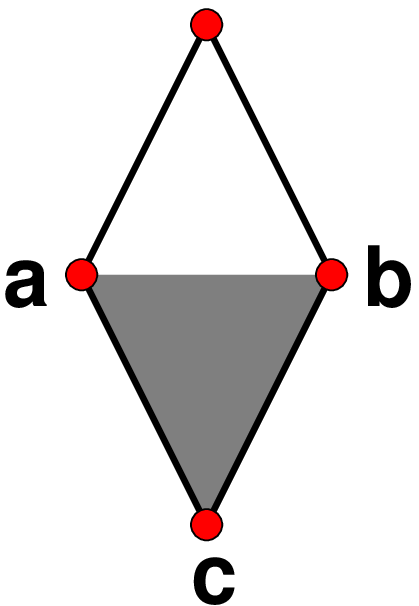}}}\to V(c,a,b)=\begin{pmatrix} \frac{a}{b} & \frac{c}{b}  \\0 & 1 \end{pmatrix}
\end{equation}
Note that the arguments $a,b,c$ of the matrices are the values of initial data attached to the three vertices of the grey lozenge.
These matrices have the remarkable property that they form a flat connection on the solutions of the $T$-system, namely we have
\begin{equation}\label{flat}V(u,a,b)\, U(b,c,v)=U(a,x,v)\, V(u,x,c)\qquad {\rm iff}\qquad xb=ac+uv \end{equation}
For some fixed integer $r$ (large enough), we may embed the above matrices in $r$-dimensional space, by defining the 
$(r+1)\times (r+1)$ matrices
$U_i,V_i$, equal to the $(r+1)\times (r+1)$ identity matrix, with the central $2\times 2$ block with row and column labels $i,i+1$ replaced by $U,V$.
The position $i$ corresponds to the $i$ coordinate of the two middle vertices of the corresponding lozenge.
We now associate to each ``slice" $S=[0,r+1]\times [j_0,j_1]$ of the initial data surface namely with vertices $\{(i,j,k_{i,j})_{i,j\in S}\}$
the matrix $M(j_0,j_1)$ equal to the product over all lozenge matrices of the slice, taken in the order of appearance of the lozenges from left to right.
This matrix is independent of the choices pertaining to the tetrahedron ambiguity above. 

Given any $j,k$ with $k\geq k_{1,j}$, let us define the left (resp. right) projections $j_0$ (resp. $j_1$) onto the stepped surface $\bk$
to be the largest (resp. smallest) integer $\ell$ such that $j-\ell=k-k_{1,\ell}$ (resp. $j-\ell=k_{1,\ell}-k$). This is illustrated below:
$$\raisebox{0.0cm}{\hbox{\epsfxsize=10.cm \epsfbox{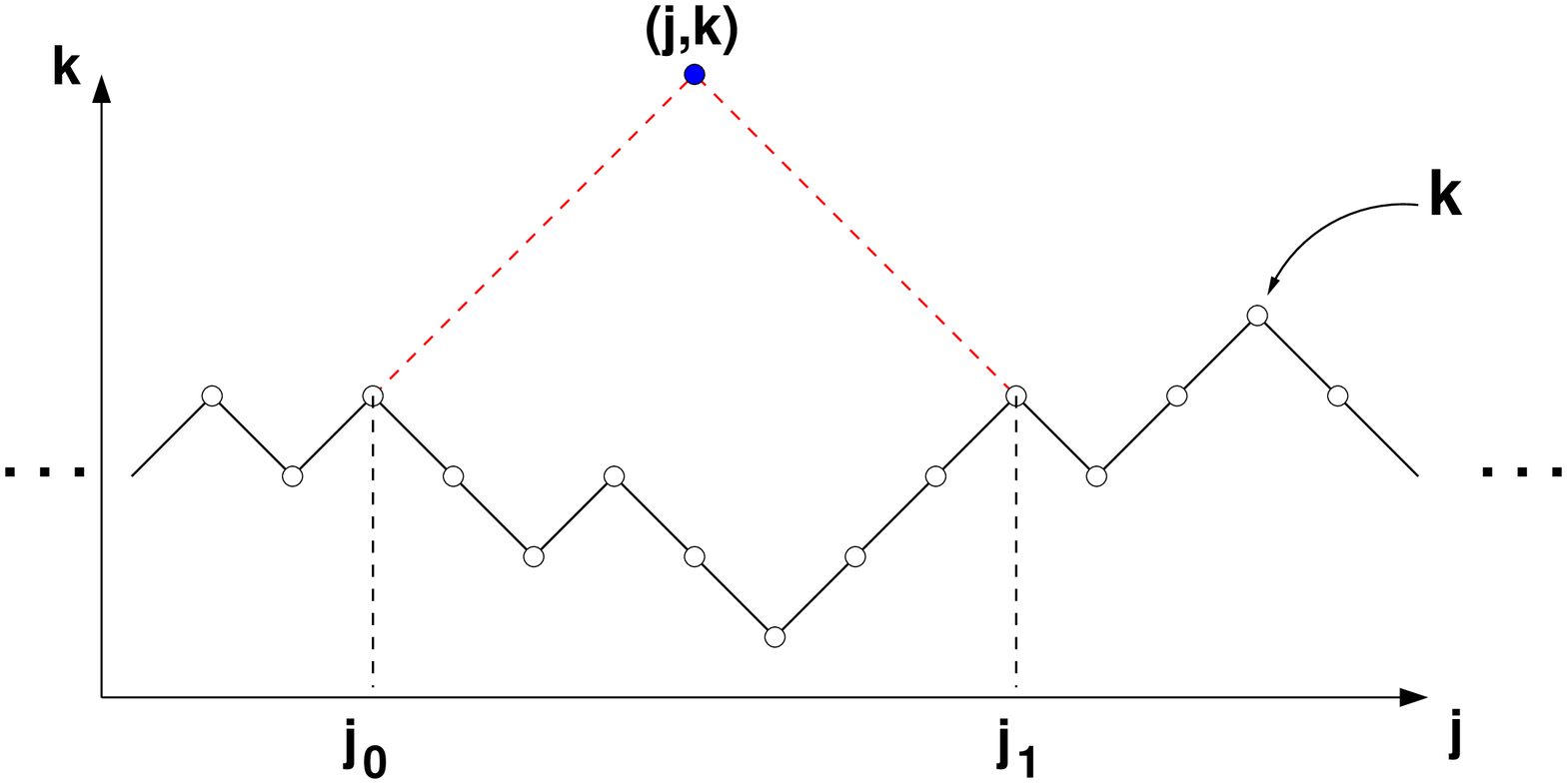}}} $$
where we have represented by empty dots the vertices of the intersection of the stepped surface $\bk$ with the $i=1$ plane.
Finally the solution $T_{j,k}$ may be written
as:
\begin{equation} T_{j,k}=M(j_0,j_1)_{1,1} \, t_{1,j_1}\end{equation}
independently of $r$ for $r$ large enough. The proof in \cite{DF} relies on the flatness condition \eqref{flat}.
Note that the positive Laurent property for $T_{j,k}$ as a function of the initial data $\{t_{i,j}\}$ is manifest,
as the entries of the matrices $U,V$ are themselves Laurent monomials of the initial data with coefficient 1 or 0.

\subsection{Network interpretation}

The matrix $M(j_0,j_1)$ can be interpreted as the weighted adjacency matrix of some oriented graph $N(j_0,j_1)$
(referred to as a ``network"), constructed as follows. First interpret the matrices $U_i,V_i$ as elementary ``chips":
\begin{equation}\label{rep2UVi}
U_i(a,b,u)=
\raisebox{-.4cm}{\hbox{\epsfxsize=2.2cm \epsfbox{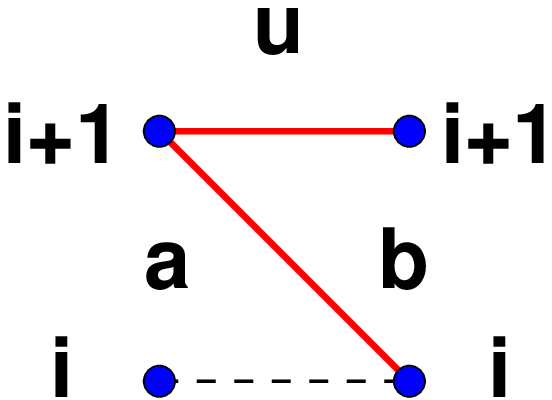}}}\qquad V_i(v,a,b)=
\raisebox{-.9cm}{\hbox{\epsfxsize=2.2cm \epsfbox{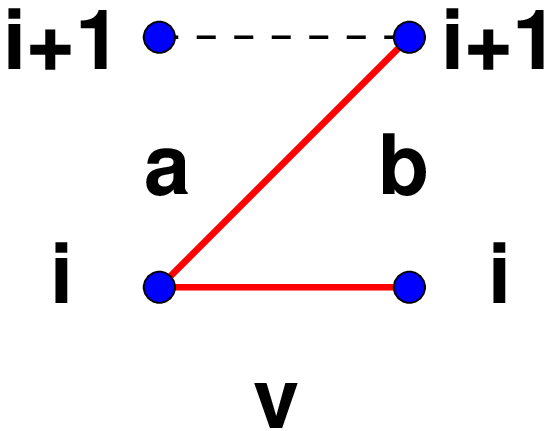}}}
\end{equation}
with two entrance connectors $i,i+1$ on the left and two exit connectors $i,i+1$ on the right, and such that the $(x,y)$ matrix
element is the weight of the oriented edge form entry connector $x$ to exit connector $y$ (here by convention all edges 
will be oriented from left to right). Note we have represented by a dashed line the ``trivial" oriented edges with weight $1$.
The top edge in the $U_i(a,b,u)$ chip has therefore weight $a/b$, the diagonal one $u/b$, etc. The arguments $a,b$, etc. appear
as face variables in the network. Any product of such matrices can be interpreted as a larger network, obtained by concatenating
the corresponding chips. We call $N(j_0,j_1)$ the network corresponding to the matrix $M(j_0,j_1)$.
Now we may interpret the matrix element $M(j_0,j_1)_{1,1}$ as the partition function for paths on $N(j_0,j_1)$
from the leftmost entry connector $1$ to the rightmost exit connector $1$. This give a nice explicit combinatorial description of 
the solution $T_{j,k}$ in terms of the initial data $\{t_{i,j}\}$, which form the face labels of the network $N(j_0,j_1)$. This is summarized in
the following:

\begin{thm}\label{netsol}
The solution $T_{1,j,k}=T_{j,k}$ of the $A_{\infty/2}$ $T$-system with initial data $(\bk,\bt)$ is $t_{1,j_1}$ times 
the partition function of paths from 
the entry connector $1$ to exit connector $1$ on the network $N(j_0,j_1)$, where $j_0,j_1$ are the left/right projections of $(j,k)$
onto the onto the stepped surface $\bk$.
\end{thm}

In \cite{DF}, it was further shown that $T_{i,j,k}$ for $i\geq 2$ is proportional to the partition functions of $i$ non-intersecting
paths on the network $N(j_0,j_1)$. As such, it enjoys the positive Laurent property as well.

\begin{example}
The network $N$ corresponding to a sample slice $S$ of flat initial data surface (see Example \ref{flatex}) reads:
$$ \raisebox{1.1cm}{\hbox{\epsfxsize=14.cm \epsfbox{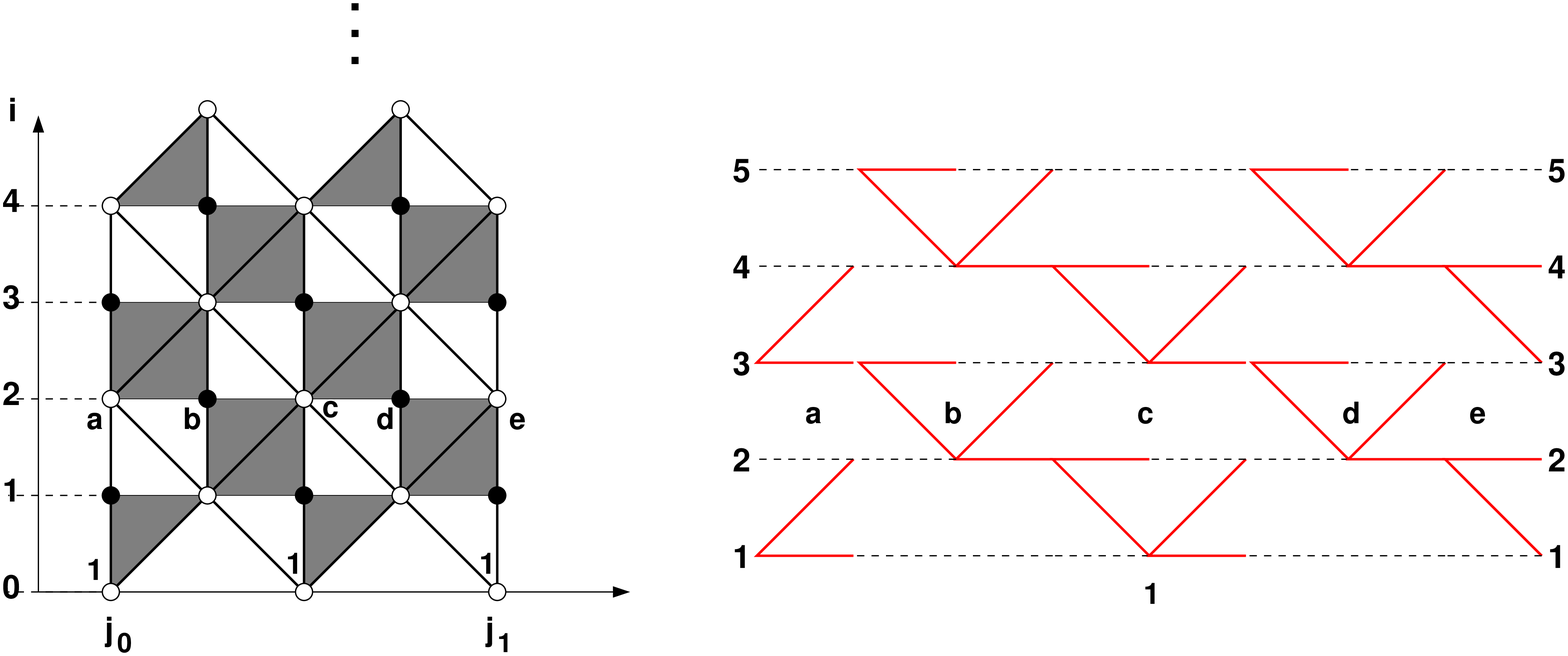}}}    $$
where we have represented the face labels in a sample row. Note also that due to the $A_{\infty/2}$
boundary condition \eqref{plus}, all the vertex/face values on the bottom row/lower face are equal to $1$.
The quantity $T_{j,k}/T_{j_1,k_{1,j_1}}$ is the partition function for paths from entry connector $1$ to exit connector $1$
on the network. We show here a sample such path, together with its local step weights:
$$ \raisebox{1.1cm}{\hbox{\epsfxsize=8.cm \epsfbox{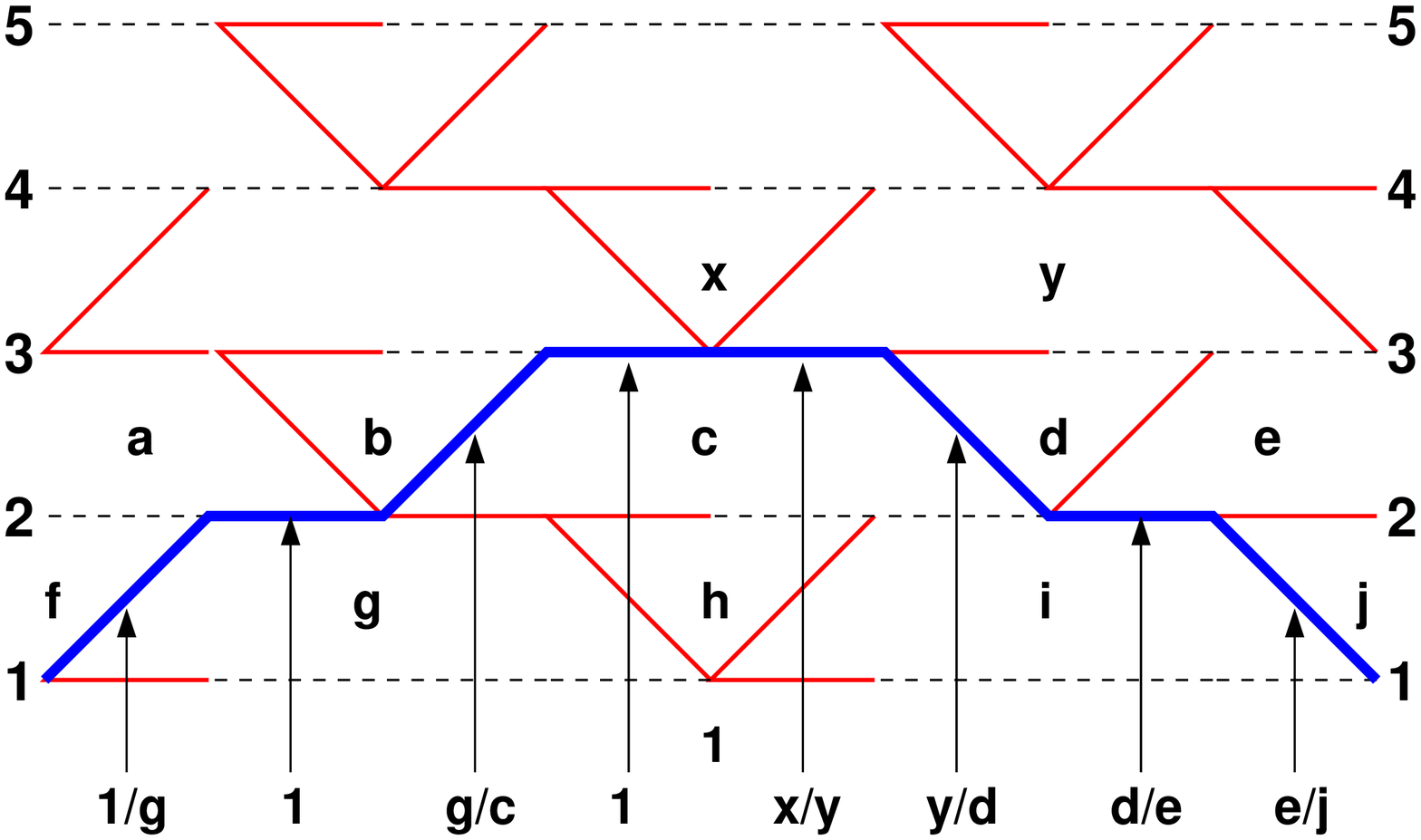}}}    $$
The total  contribution of this path to the partition function is therefore: 
$\frac{1}{g}\frac{g}{c}\frac{y}{x}\frac{x}{d}\frac{d}{e}\frac{e}{j}=\frac{x}{c j}$
\end{example}

\section{Application: computation of the Laurent polynomials $p_{a,b}(\theta_\cdot)$}\label{applione}

In this section, we show that the functions $p_{a,b}(\theta_\cdot)$ defined by the system \eqref{condet} are the solutions $T_{j,k}=T_{1,j,k}$ of 
the $A_{\infty/2}$ $T$-system with some particular initial conditions and some particular mapping of indices $(a,b)\to (j,k)$.

\subsection{Steepest stepped surface}\label{steepest}

In this section we consider stepped surfaces with fixed intersection with the $i=1$ plane, say equal to a path $\pi=(j,k_j)$ with 
$|k_{j+1}-k_j|=1$ and $j+k_j=0$ mod 2 for all $j\in \Z$.  We define the ``steepest stepped surface" $\bk_\pi$ 
to be the unique stepped surface $\bk$ with $k_{1,j}=k_j$ included in the union of all the planes with normal vector $(1,-1,-1)$
or $(1,1,-1)$ and passing through pairs of distinct points of $\pi$. It is easy to see that such a surface is piecewise-linear.
Let $(a_j,b_j=k_{a_j})$, $j\in \Z$ be the vertices where $\pi$ changes direction, with say a minimum when $j$ is even, and a 
maximum when $j$ is odd. The steepest stepped surface $k=k_{i,j}$ is defined by the following equations for $m\in \Z$:
\begin{eqnarray*}
i-j-k_{i,j}&=&1-a_{2m}-b_{2m}\qquad (i\geq 1;a_{2m-1}\leq j\leq a_{2m}) \\
i+j-k_{i,j}&=&1+a_{2m}-b_{2m}  \qquad (i\geq 1;a_{2m}\leq j\leq a_{2m+1})
\end{eqnarray*}

Concretely, the steepest stepped surface is a sort of roof of fixed slope above the infinite Young diagram delimited by the path $\pi$
in the $i=1$ plane (see Fig.\ref{fig:toit} for a 3D view in perspective).

\subsection{Connection with the Laurent polynomials $p_{a,b}(\theta_\cdot)$}

\begin{figure}
\centering
\includegraphics[width=16.cm]{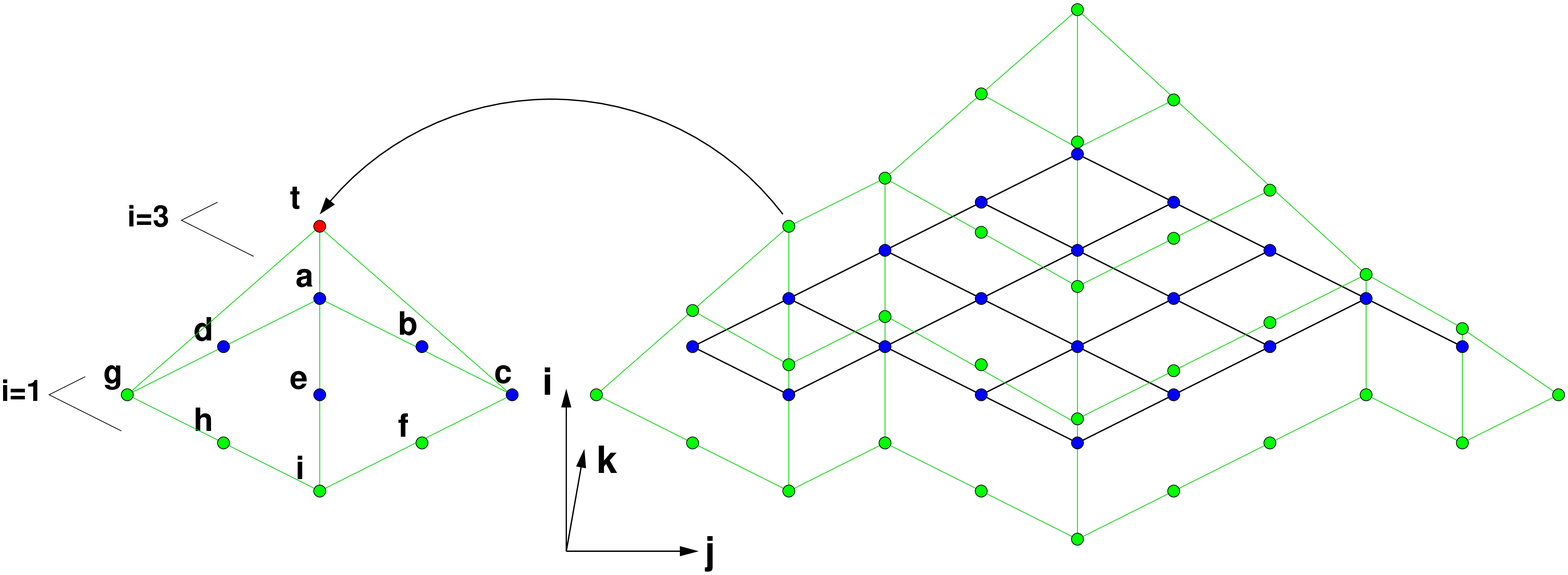}
\caption{\small The FCC lattice representation of the Young diagram $\lambda=(5,4,4,4,2)$ (blue dots in the $i=1$ plane)
and its augmented tableau $\lambda^*$ (extra green vertices in the $i=1$ plane), together with the initial data steepest stepped surface
(green vertices, including $\lambda^*\setminus\lambda$). We have indicated a particular vertex of the surface
in the plane $i=3$, with assigned initial value $t$, and displayed the pyramid of which it is the apex, together with its base in the $i=1$ plane, 
identified as a square of size $3$ of $\lambda^*$. The initial data assignment amounts to imposing that the $3\times 3$ determinant of the values
of $T$ in this square is equal to the value $t$ at the apex.}
\label{fig:toit}
\end{figure}

Let us represent the centers of boxes of the Young diagram $\lambda^*$ as vertices of the $i=1$ plane of the lattice $L_{FCC}$
(see Fig. \ref{fig:toit} for an illustration).
This allows to identify the strip $\lambda^*\setminus \lambda$ as a path $p$ from $(0,0)$ to $(N+\lambda_1,N-\lambda_1)$, while
the NW corner box of the diagram has coordinates $(N,N)$ (see sketch below):
$$ \raisebox{1.1cm}{\hbox{\epsfxsize=11.cm \epsfbox{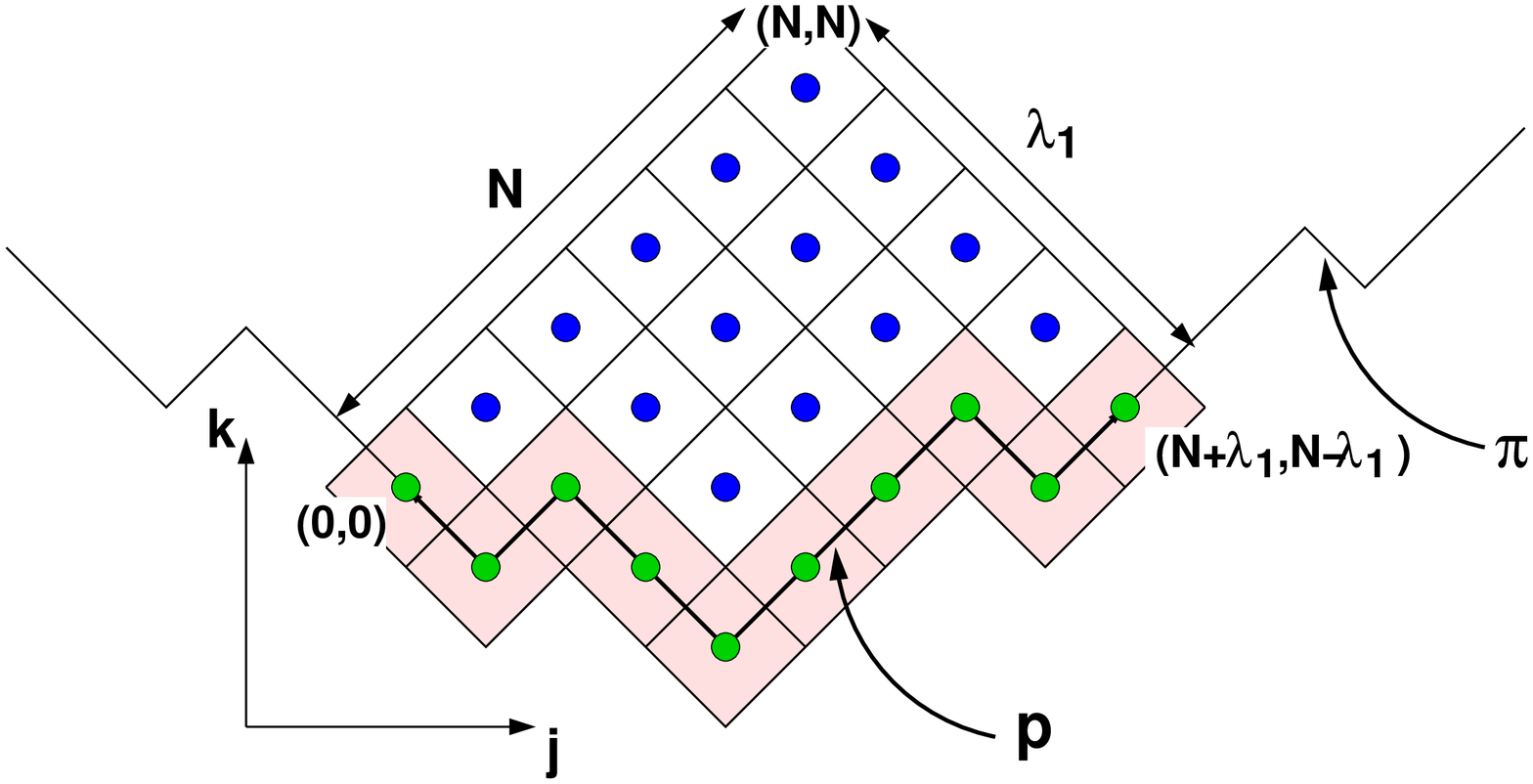}}}    $$
Let us extend arbitrarily the path $p$ into a path $\pi$ on the entire plane, and consider the solutions $T_{j,k}$ of the 
$A_{\infty/2}$ $T$-system with initial data stepped surface equal to the steepest surface $\bk_\pi$ associated to $\pi$,
and with initial values $t_{i,j}$. To avoid confusion, note that there are distinct yet natural frames used so far for expressing
the coordinates of the centers of the boxes of $\lambda^*$. 
On one hand, we have the original frame, in which the coordinate $(a,b)$ refers
to the box in row $a\in[1,N]$ and column $b\in [1,\lambda_a]$. On the other hand we have the $T$-system (or FCC lattice) frame
in the plane $i=1$, in which the coordinate $(j,k)$ of the center of a box is related to that in the original frame by:
\begin{equation}j=\lambda_1+a-b,\qquad k=\lambda_1+2-a-b 
\end{equation}
In the following, we will use letters $a,b$ for the original coordinates, and $j,k$ for the FCC lattice coordinates.

We have the following:

\begin{thm}\label{pT}
The system \eqref{condet} has a unique solution $p_{a,b}(\theta_\cdot)$, $(a,b)\in \lambda$. Moreover, we have
\begin{equation}
p_{a,b}(\theta_\cdot)=T_{a-b+\lambda_1,\lambda_1+2-a-b}\qquad \big((a,b)\in \lambda^*\big)
\end{equation}
where $T_{j,k}=T_{1,j,k}$ is the solution of the $A_{\infty/2}$ $T$-system with steepest initial data stepped surface $\bk_\pi$
and initial values $\bt$ such that
\begin{equation}
\theta_{a,b}=t_{a-b+\lambda_1,\frac{\lambda_1+2-a-b+k_{a-b+\lambda_1}}{2}}\qquad \big((a,b)\in \lambda^*\big)
\end{equation}
where $(j,k_j)$ are the vertices of the path $p$, for $j=0,1,...,\lambda_1+N$.
\end{thm}
\begin{proof}
Consider the solution $T_{i,j,k}$ of the  $A_{\infty/2}$ $T$-system with steepest initial data stepped surface $\bk_\pi$
and initial values $\bt$. Let us restrict our attention to the solutions $T_{j,k}=T_{1,j,k}$ at the points $(j,k)$ with left and right projections
on the interval $[0,\lambda_1+N]$. These are exactly the centers of the boxes of $\lambda^*$ in the above representation
(see Fig.\ref{fig:toit} for the case $\lambda=(5,4,4,4,2)$).

On the other hand, the equation \eqref{wrons} may be interpreted as follows. The indices $(j+b-a,k+i+1-a-b)$
for $1\leq a,b \leq i$ are the coordinates in the plane $x=1$ of the base of the pyramid $\Pi_{i,j,k}$ with apex $(i,j,k)$, defined as
$\Pi_{i,j,k}=\{ (x,y,z)\in L_{FCC},\,  |y-j|+|z-k|\leq |x-i|\}$. As illustrated in Fig.\ref{fig:toit}, each vertex $(i,j,k_{i,j})$ of the steepest surface 
$\bk_\pi$ is the apex of 
such a pyramid $\Pi_{i,j,k_{i,j}}$.  By definition of the steepest stepped surface,
the base of the pyramid $\Pi_{i,j,k_{i,j}}$ in the $x=1$ plane is the square $S_{a,b}$ of $\lambda^*$, with NW corner at position
$(j,2k-k_j)=(a-b+\lambda_1,\lambda_1+2-a-b)$ (and SE corner at position $(j,k_j)$ on the path $p$), hence $a=1-k_{i,j}+\frac{j+k_j}{2}$
and $b=\lambda_1+1+k_{i,j}+\frac{k_j-j}{2}$. We conclude that $T_{i,j,k_{i,j}}=t_{i,j}$ is the determinant of the array $T_{\ell,m}$ for $(\ell,m)$
in the square $S_{a,b}$. In the example of Fig.\ref{fig:toit} (left), this amounts to the identity 
$t={ \tiny{\left\vert \begin{matrix}a & b & c\\ d & e & f\\ g & h & i\end{matrix}\right\vert}}$. 

Let us identify $t_{i,j}\equiv T_{i,j,k_{i,j}}$ with $\theta_{a,b}$, for all $(a,b)\in \lambda^*$,
with $a=1-k_{i,j}+\frac{j+k_j}{2}$
and $b=\lambda_1+1+k_{i,j}+\frac{k_j-j}{2}$. Then the system of equations \eqref{condet} for $p_{a,b}$ becomes the {\it same} as that for
$T_{j,k}$, with $j=a-b+\lambda_1$ and $k=\lambda_1+2-a-b$. 
As all $T_{j,k}$ are uniquely determined by the initial data, then so
are the $p_{a,b}$, and the theorem follows.
\end{proof}


\subsection{Network interpretation}\label{networksec}

We may now specialize the general solution of Sect.\ref{tsysect} to the case of the steepest stepped surface. Let us describe the 
extended Young tableau $\lambda^*$ via the sequence of integers $n_1,m_1,n_2,m_2,...,n_k,m_k$
corresponding to the length of straight portions of the path $p$ delimiting the diagram, namely:
$$ \raisebox{1.1cm}{\hbox{\epsfxsize=7.cm \epsfbox{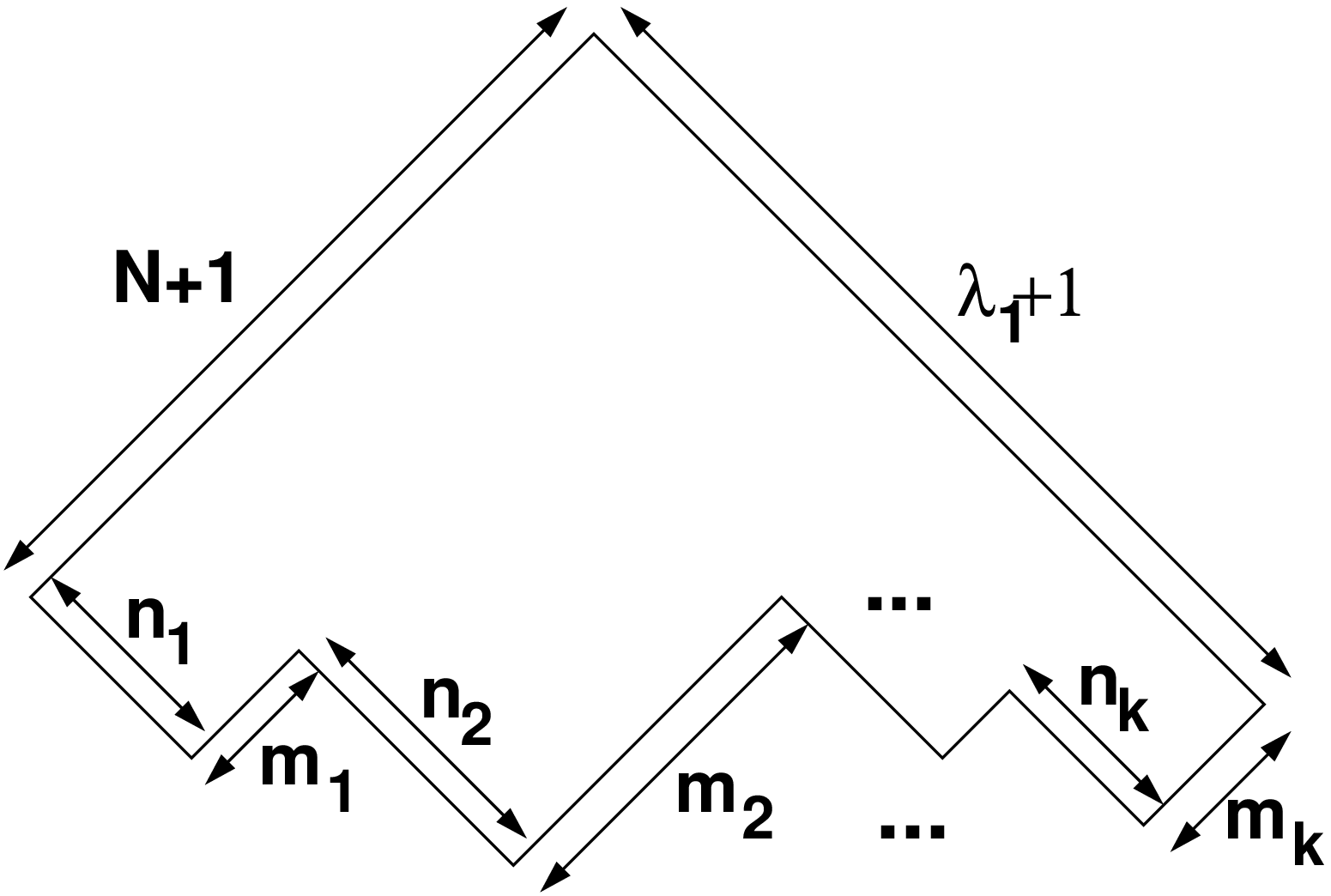}}}     $$
or in the notations of Sect. \ref{steepest}: $n_{i}=a_{2i}-a_{2i-1}$, $m_i=a_{2i+1}-a_{2i}$, $i=1,2,...,k$,
where $a_1,a_2,...,a_{2k}$ are the $j$ coordinates in the FCC lattice of the changes of slope of $p$.

A drastic simplification occurs in the case of the steepest stepped surface: along each steepest plane,
only one type of lozenge $U$ or $V$ occurs, namely planes orthogonal to $(1,-1,-1)$ have a lozenge decomposition
using only $V$ type lozenges, while those orthogonal to $(1,1,-1)$ have a lozenge decomposition
using only $U$ type lozenges. Moreover the slice of surface corresponding to $T_{i,j,k}$ with lower/upper projections $j_0,j_1$
always starts with a $(1,-1,-1)^\perp$ plane and ends with a $(1,1,-1)^\perp$ one. The corresponding lozenge decomposition
reads typically like:
$$ \raisebox{1.1cm}{\hbox{\epsfxsize=11.cm \epsfbox{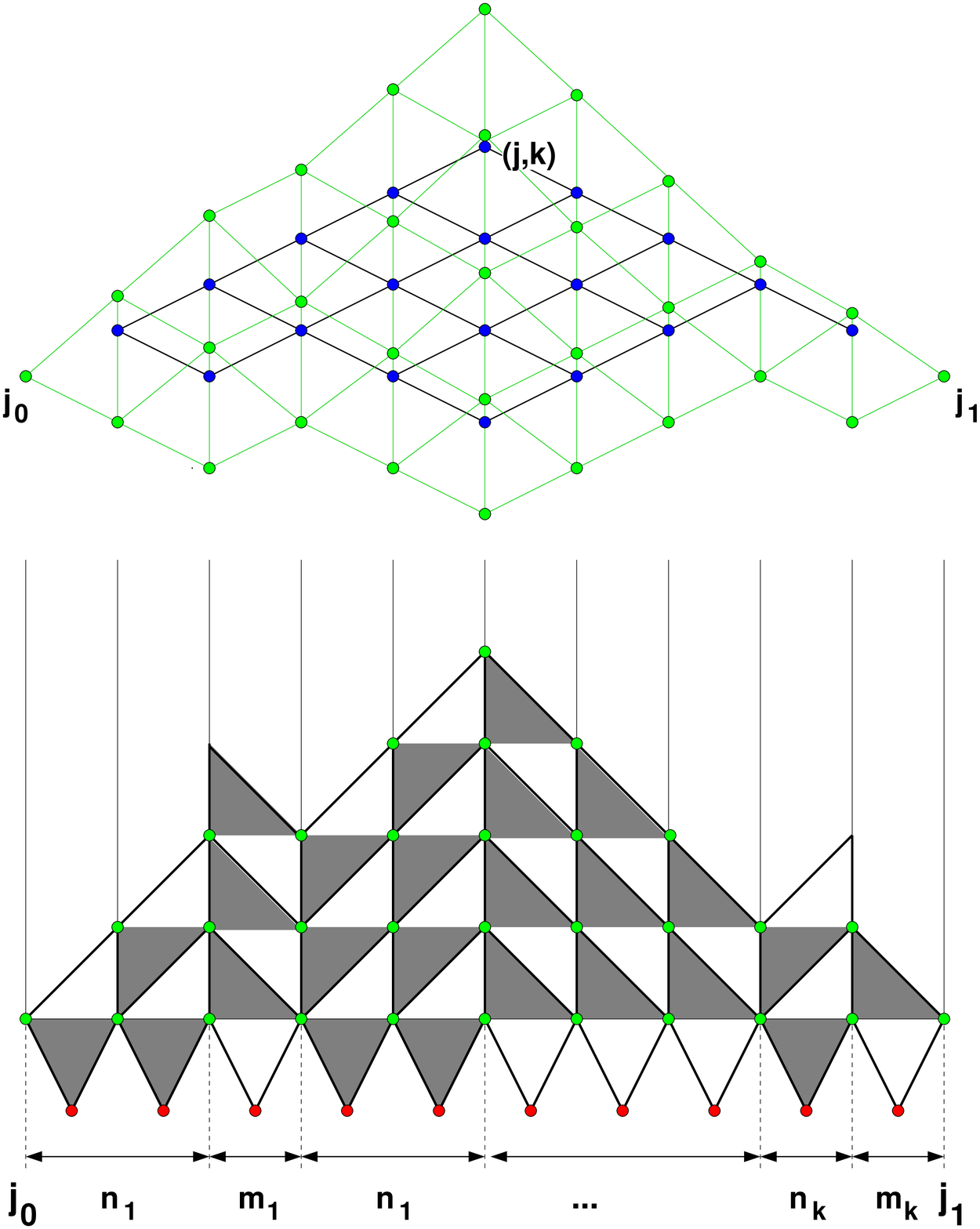}}}      $$
where we have only included the minimal number of $V$ or $U$ type lozenges in each slice (any extra $U,V$ would have no effect
on the corresponding matrix element $M(j_0,j_1)_{1,1}$). Recall that the vertices of the surface carry initial data assignments $t_{i,j}$,
in bijection with the $\theta_{a,b}$ parameters, while the bottom layer at $i=0$ carries values all equal to $1$.

We may now construct the network associated to the lozenge decomposition above. For our running example, it reads:
$$ \raisebox{0.cm}{\hbox{\epsfxsize=17.cm \epsfbox{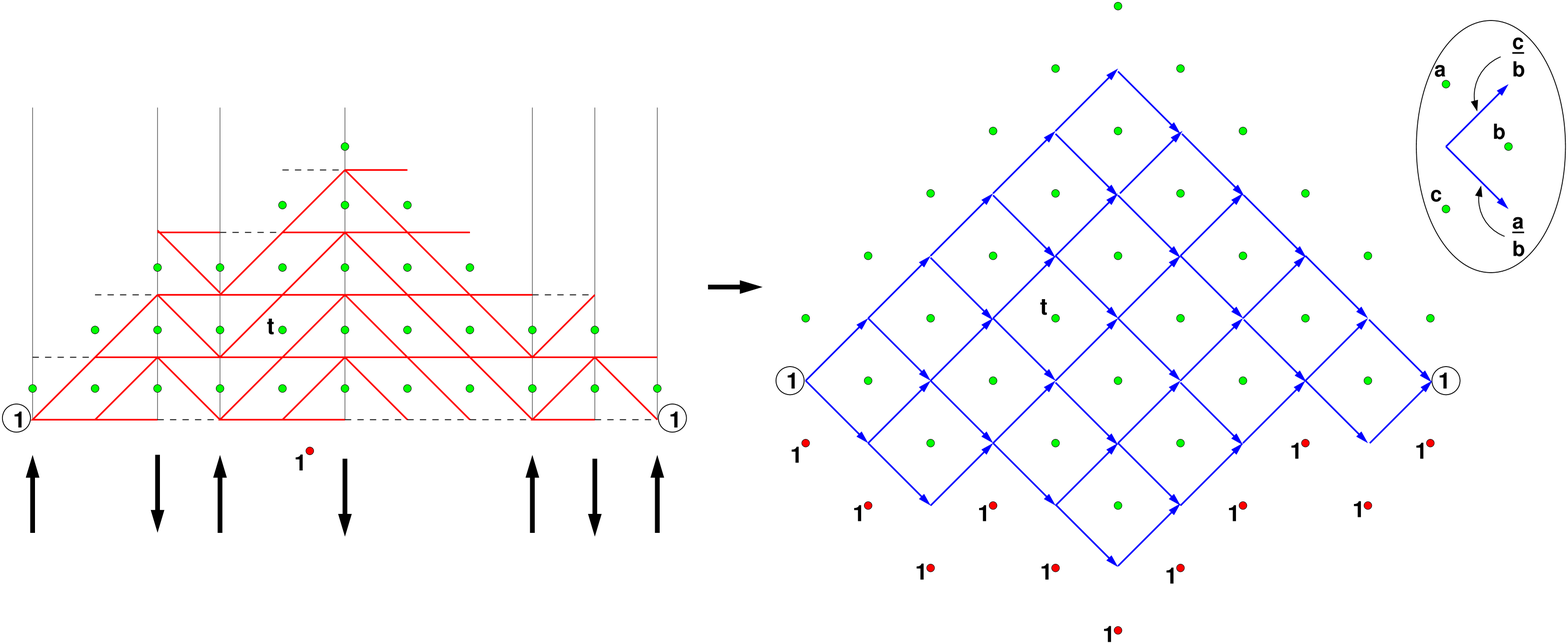}}}      $$
where we have indicated how to deform the original network graph to bring it to the square lattice with oriented edges. 
This latter graph is denoted by $L(j_0,j_1)$ in FCC lattice language. Assuming that the top inner face of $L(j_0,j_1)$
corresponds to the box $(a,b)$ of $\lambda$, we denote alternatively this network by ${\mathcal N}_{a,b}\equiv L(j_0,j_1)$. 
We have indicated the face variables of this network by green dots (the top vertex carries the variable $\theta_{a,b}$), 
while all variables on the (bottom-most) red dots are equal to $1$.
The medallion summarizes
the weighting rule for the edges of  $L(j_0,j_1)={\mathcal N}_{a,b}$ in terms of the face variables $t_{i,j}=\theta_{a,b}$ (at the indicated vertices). 
This is just a rephrasing of the weights of $U$ and $V$ type chips after the above deformation, namely:
\begin{equation}\label{rulnet} U=\raisebox{-1.cm}{\hbox{\epsfxsize=4.cm \epsfbox{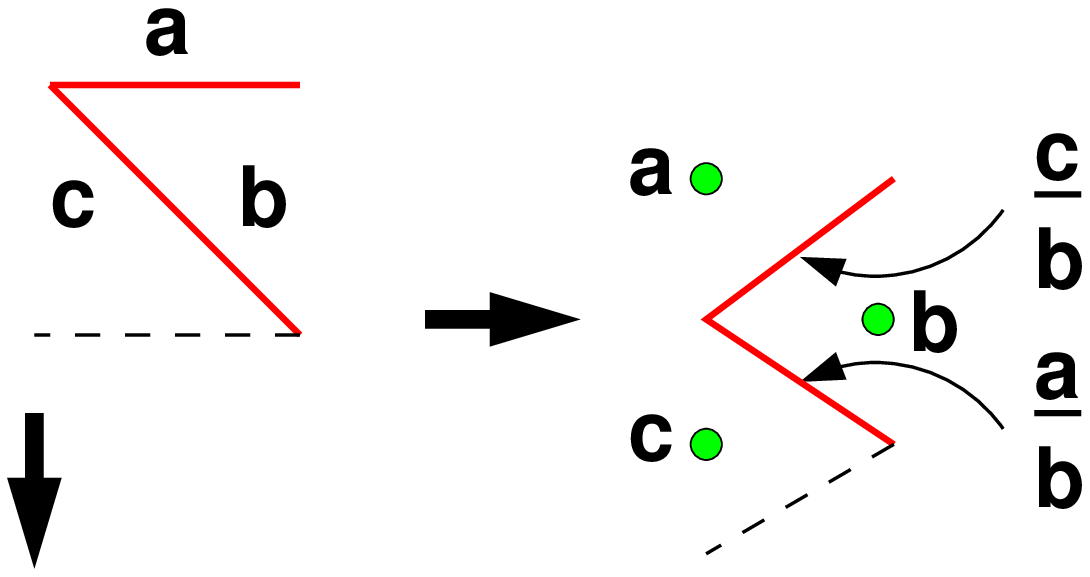}}},
\qquad V= \raisebox{-1.cm}{\hbox{\epsfxsize=4.cm \epsfbox{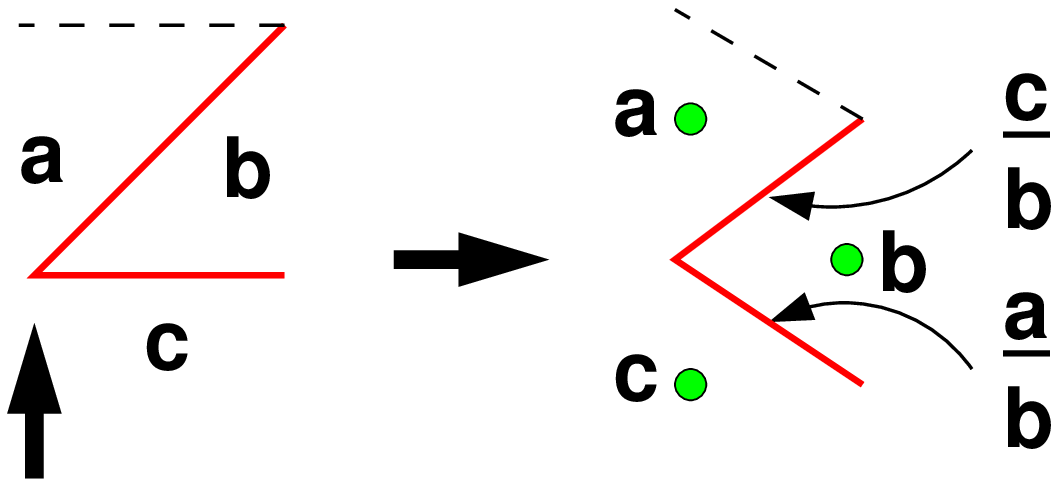}}}
\end{equation} 
The graph $L(j_0,j_1)={\mathcal N}_{a,b}$ is nothing but the actual initial Young diagram 
$\lambda_{a,b}$ (SE corner at $(a,b)$) represented tilted by $45^\circ$, 
and with all box edges oriented from left to right. The face variables,
expressed indifferently as $t_{i,j}$'s or $\theta_{a,b}$'s, are actually at the centers of boxes of $\lambda_{a,b}^*$ 
but displaced by a global translation of $(2,0)$ in the $(j,k)$ plane 
(or by one column to the left and one row up in the original frame). For convenience we denote by $\lambda_{a,b}^\#$
this latter displaced array.
Note that extra variables equal to $1$ occupy the centers of boxes of $\lambda_{a,b}^*\setminus \lambda_{a,b}$.
In the above depiction, we have $\lambda_{a,b}=(5,4,4,4,2)$
and $\lambda_{a,b}^*=(6,6,5,5,5,3)$.

By use of Theorems \ref{netsol} and \ref{pT}, the value of $p_{a,b}=T_{j,k}$ in the top box is given by the partition function for paths
on $L(j_0,j_1)={\mathcal N}_{a,b}$ from the leftmost vertex to the rightmost one, multiplied by the bottom right initial value 
$t_{1,j_1}=\theta_{a,\lambda_a^*}=\theta_{a,\lambda_a+1}$. 
We summarize this result in the following:

\begin{thm}\label{thpath}
The solution $p_{a,b}$ to the system \eqref{condet} is the partition function of paths on the weighted graph ${\mathcal N}_{a,b}$
associated to $\lambda_{a,b}$, from the leftmost to the rightmost vertex, multiplied by the variable $\theta_{a,\lambda_a+1}$ of the top
right box of $\lambda_{a,b}^*$.
\end{thm}

Note that Theorem \ref{main} follows from this result, as the sum over paths produces a manifestly positive Laurent polynomial of the initial parameters $\theta_{i,j}$.

\begin{example}
Let us revisit Example \ref{firstex}. The graph $L(j_0,j_1)={\mathcal N}_{1,1}$ for the calculation of $p_{1,1}$ is associated to 
$\lambda_{1,1}=\lambda=(2,1)$, and reads:
$$\raisebox{0.cm}{\hbox{\epsfxsize=12.cm \epsfbox{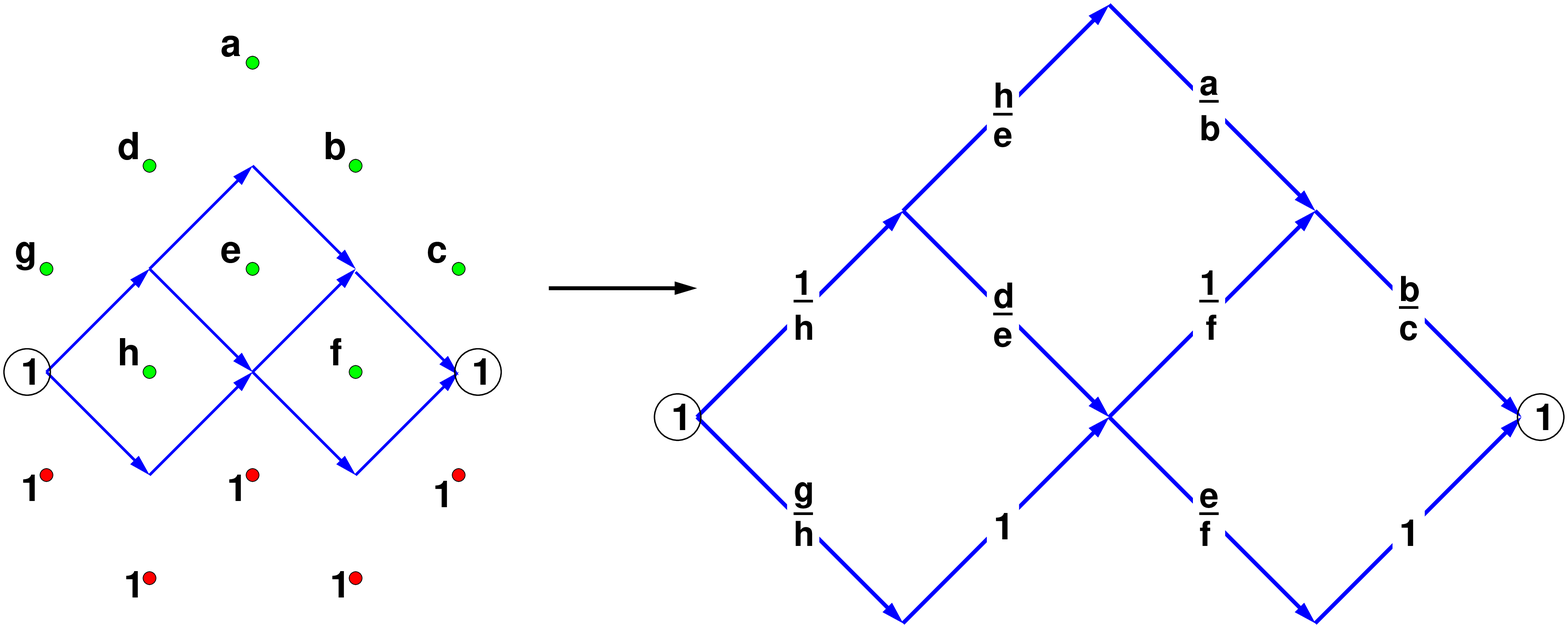}}} $$
where we have indicated the edge weights.
The Laurent polynomial $p_{1,1}$ is $c$ times the partition function for the 5 paths $1\to 1$:
$$\begin{matrix}\begin{matrix}{\rm path:} \\ {} \end{matrix}& \hbox{\epsfxsize=2.cm \epsfbox{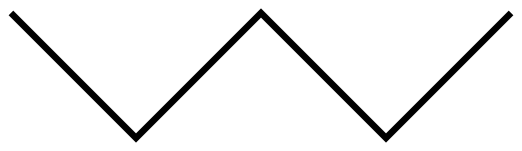}} &  \hbox{\epsfxsize=2.cm \epsfbox{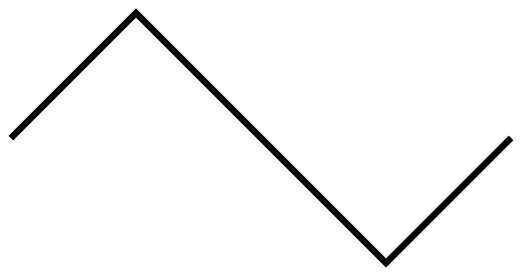}}& 
\hbox{\epsfxsize=2.cm \epsfbox{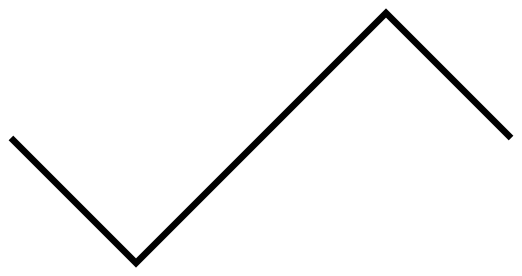}} &  \hbox{\epsfxsize=2.cm \epsfbox{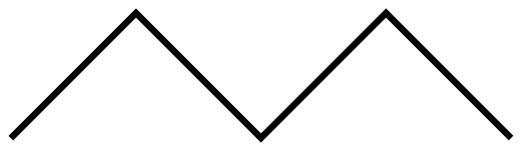}}& \hbox{\epsfxsize=2.cm \epsfbox{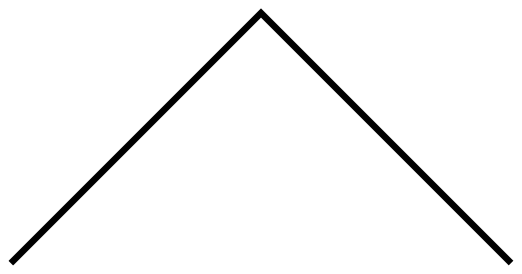}} \\
{\rm weight:} & \frac{ge}{f h c}& \frac{d}{f h} &\frac{gb}{fhc} &\frac{bd}{hefc} & \frac{a}{ce}  \end{matrix}$$
wich yields
$$ p_{1,1}=\frac{a f h+(b+c e)(d+g e)}{hef} $$
in agreement with the expression of Example \ref{firstex}.
\end{example}

\subsection{Dimer interpretation}\label{dimersec}

The network formulation for the solution of the $T$-system described in Sect.\ref{tsysect} can be rephrased in terms of a statistical model of dimers on a planar bipartite graph ${\mathcal G}$ with face variables, made only of square, hexagonal and octagonal inner faces, and with open outer faces 
adjacent to 1 or 2 edges (see Ref.\cite{DF13}). This graph is constructed as the dual of the lozenge decomposition, essentially by substituting:
$$\raisebox{-1.3cm}{\hbox{\epsfxsize=5.cm \epsfbox{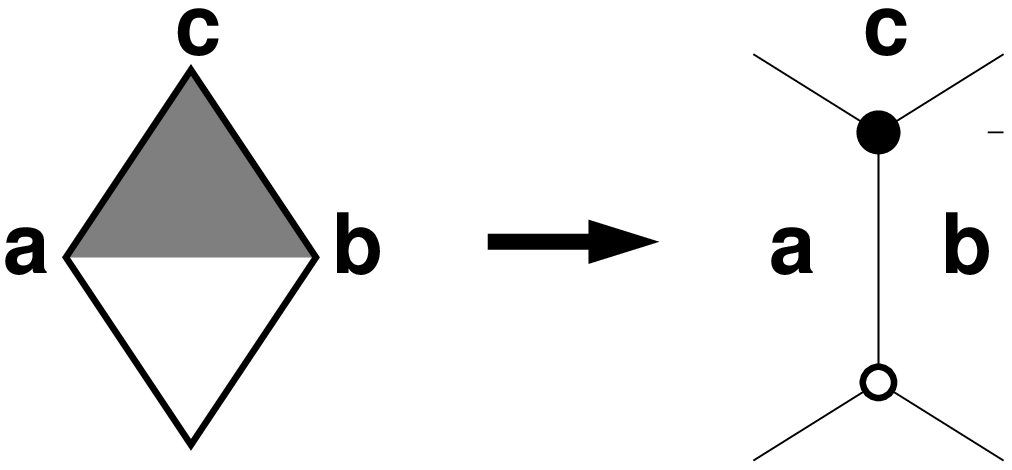}}}\quad {\rm and}\quad  \raisebox{-1.2cm}{\hbox{\epsfxsize=5.cm \epsfbox{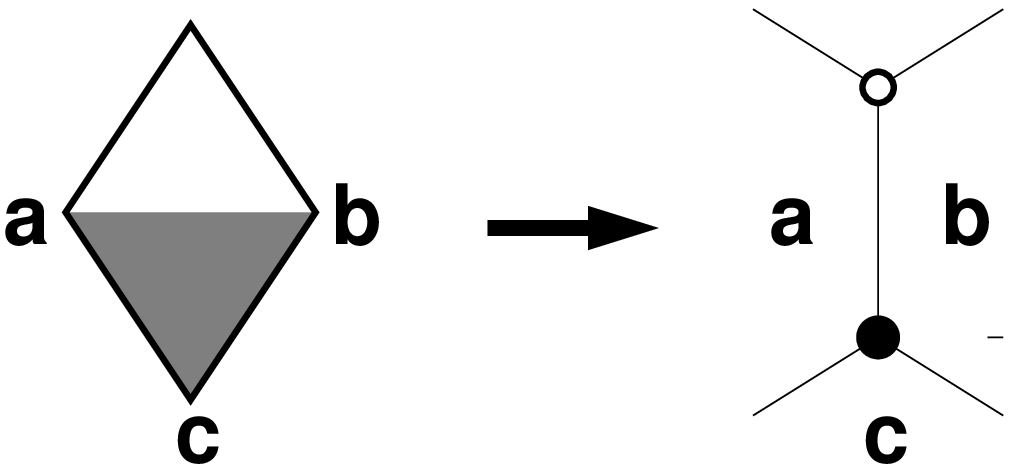}}}   $$
Note that the initial data assignments become face variables in the dimer graph.
The configurations of the dimer model on 
a bipartite graph are obtained by covering single edges of the graph by ``dimers" in such a way that each vertex is covered exactly once.
The weight of a given configuration is the product of local outer/inner face weights expressed in terms of the attached face variable.
The weight of an inner face is $a^{\frac{v}{2}-1-D}$ where $a$ is the face variable, $v$ the degree of the face ($v\in\{4,6,8\}$),
and $D$ the total number of dimers occupying edges bordering the face. The weight of an outer face is $b^{1-D}$, where $b$ is the face variable
and $D$ the total number of dimers occupying edges of the graph adjacent to the face.
Then we have:

\begin{thm}\label{thdim}{\cite{DF13}}
The solution $T_{1,j,k}\equiv T_{j,k}$ of the $A_{\infty/2}$ $T$-system is the partition function of the dimer model on the dimer graph dual to the lozenge decomposition of the corresponding network.
\end{thm}

For the particular case of the steepest stepped surface, the dimer graph ${\mathcal G}_{a,b}$ for the computation of $p_{a,b}$
is particularly simple. Its inner faces occupy a domain
of the hexagonal (honeycomb) lattice with the shape of the young diagram $\lambda_{a,b}$, while its outer faces correspond
respectively to $\lambda_{a,b}^*\setminus \lambda_{a,b}$ with face variables all equal to $1$, and to $\lambda_{a,b}^\#\setminus \lambda_{a,b}$.
The other face variables are the variables $\theta_{\alpha,\beta}$ on $\lambda_{a,b}^\#$.
For the case $\lambda_{a,b}=(5,4,4,4,2)$ of previous section, the dimer graph ${\mathcal G}_{a,b}$ reads:
$$ \raisebox{-1.3cm}{\hbox{\epsfxsize=9.cm \epsfbox{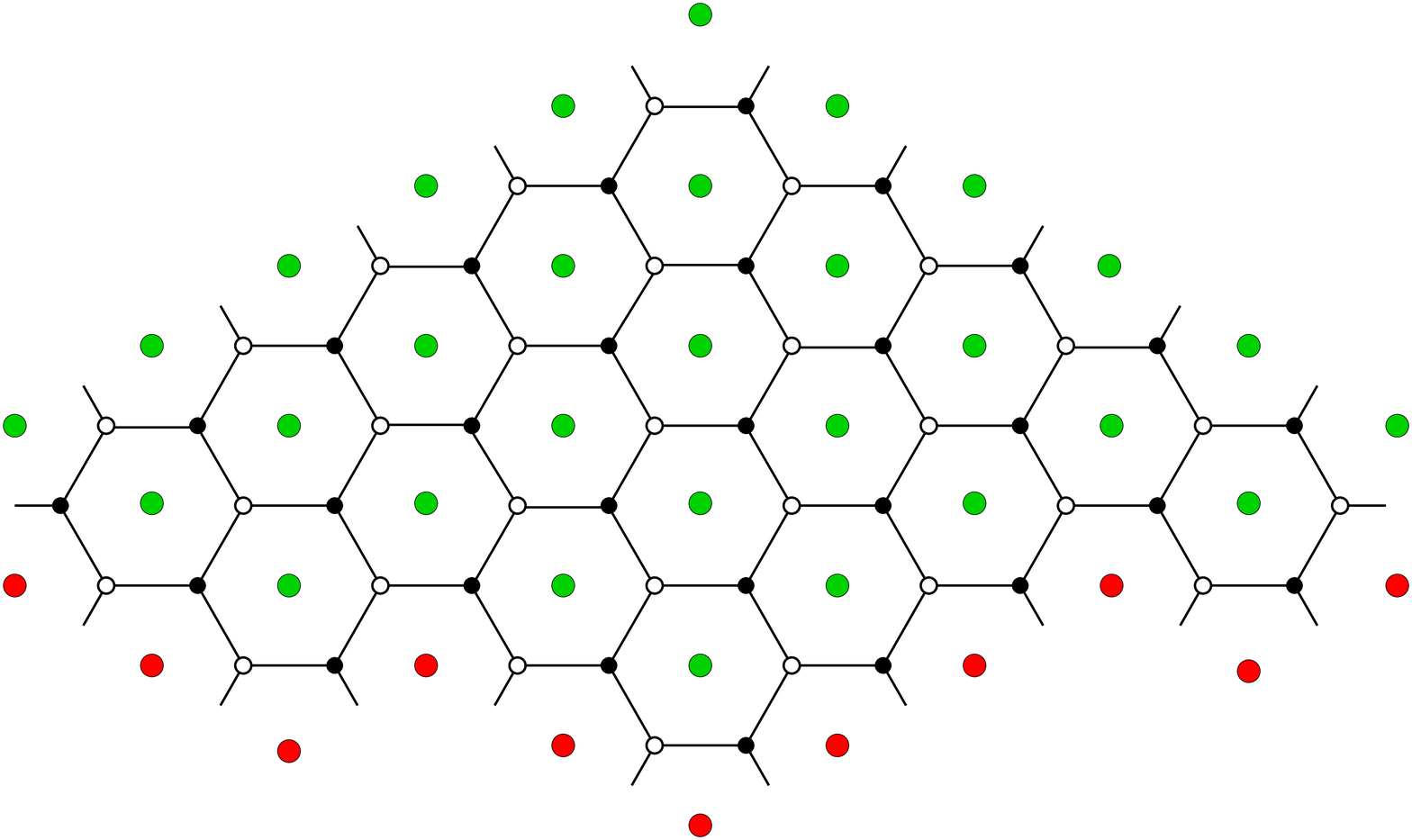}}}$$
where we have represented in red the centers of boxes of $\lambda_{a,b}^*\setminus\lambda_{a,b}$ (all with assigned face values $1$),
and in green the centers of boxes of $\lambda_{a,b}^\#$ (with the $\theta_{\alpha,\beta}$'s or $t_{i,j}$'s as assigned faces values).

Applying Theorem \ref{thdim}, we finally get:

\begin{thm}
The Laurent polynomial $p_{a,b}(\theta_\cdot)$ is the partition function for the dimer model on the graph ${\mathcal G}_{a,b}$ with face variables 
$\theta_{\alpha,\beta}$ on the faces corresponding to the boxes of $\lambda_{a,b}^\#$ and face variables $1$ on those corresponding to the boxes
of $\lambda_{a,b}^*\setminus\lambda_{a,b}$.
\end{thm}

\begin{example}
Let us revisit the example \ref{firstex}. The graph ${\mathcal G}_{1,1}$ for computing $p_{1,1}$ reads:
$$ \raisebox{-1.3cm}{\hbox{\epsfxsize=6.cm \epsfbox{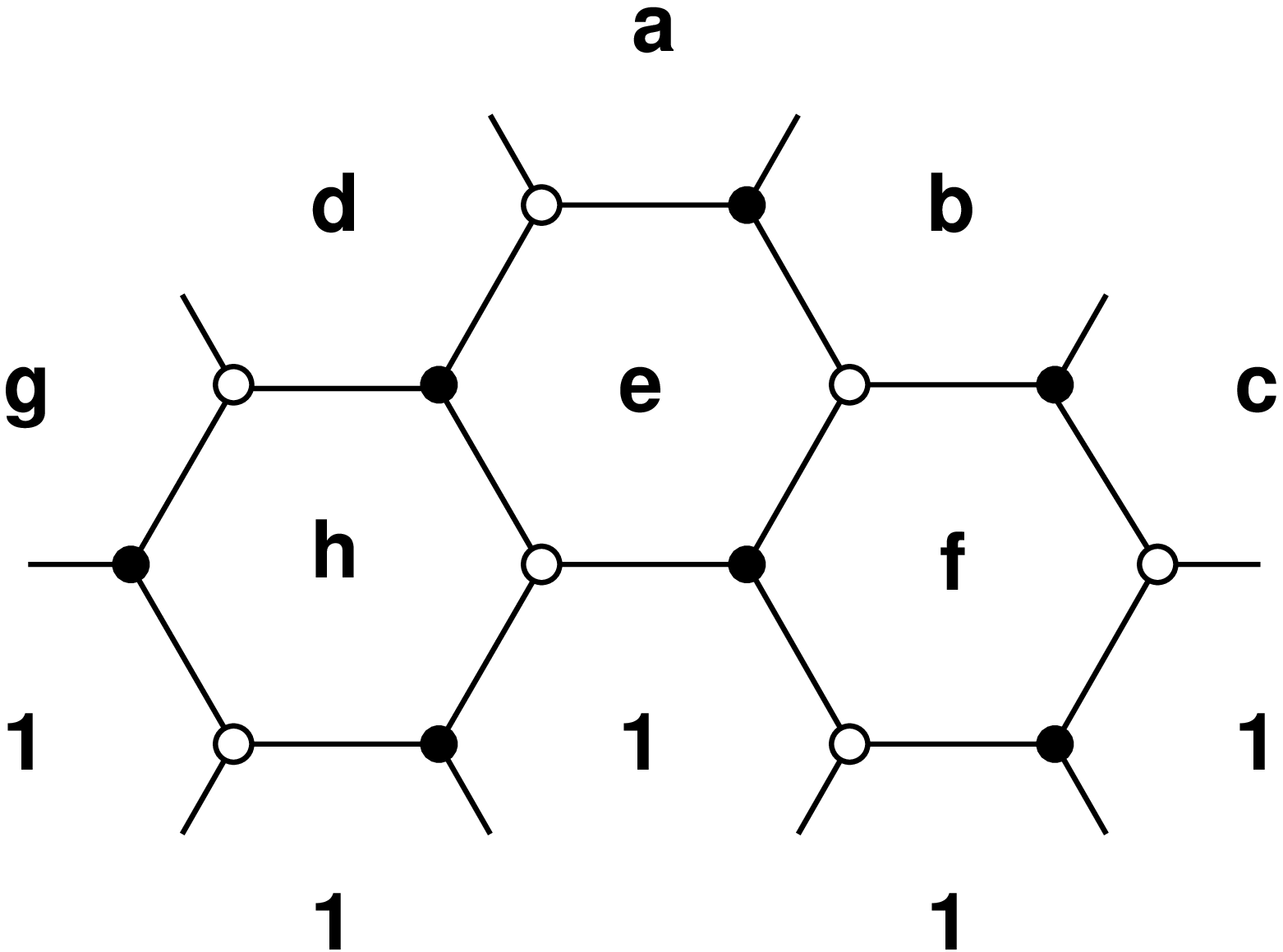}}} $$
and the partition function for dimers on ${\mathcal G}_{1,1}$ is the sum over the following five configurations:
$$\begin{matrix} \begin{matrix} {\rm dimer}\\
{\rm configuration} :\\ {}\\{} \end{matrix} &\raisebox{-.3cm}{\hbox{\epsfxsize=2.5cm \epsfbox{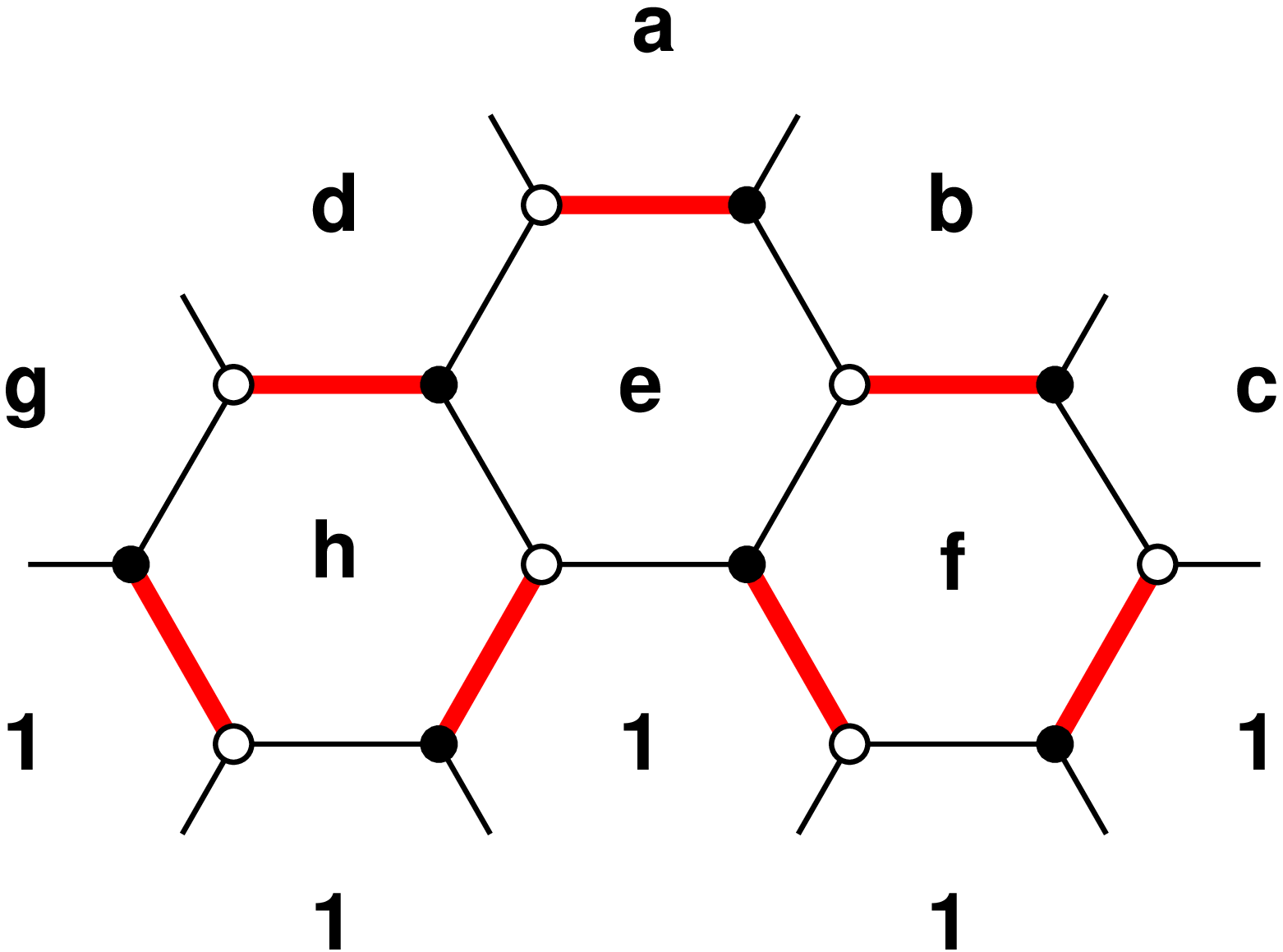}}} 
& \raisebox{-.3cm}{\hbox{\epsfxsize=2.5cm \epsfbox{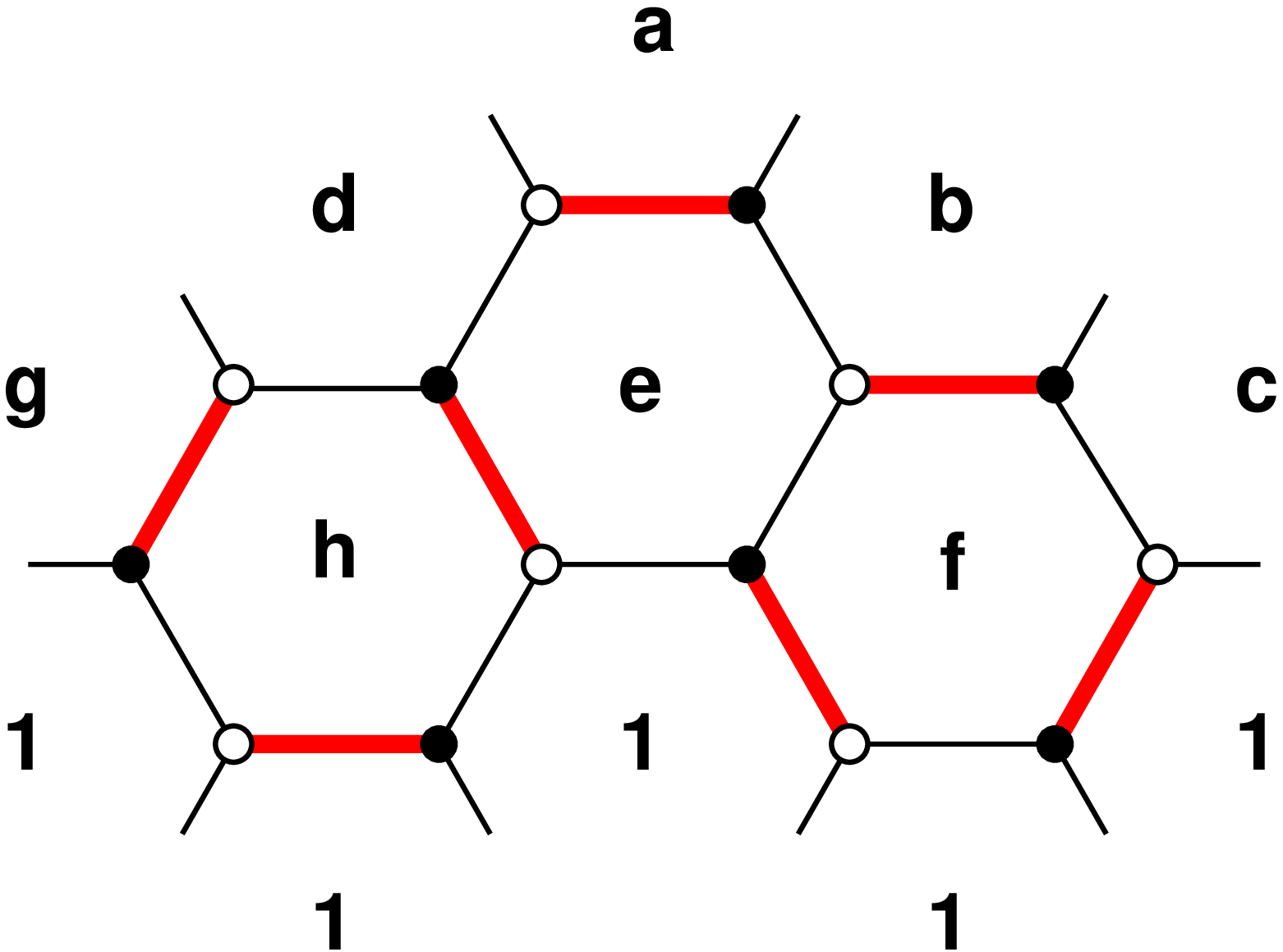}}}& 
\raisebox{-.3cm}{\hbox{\epsfxsize=2.5cm \epsfbox{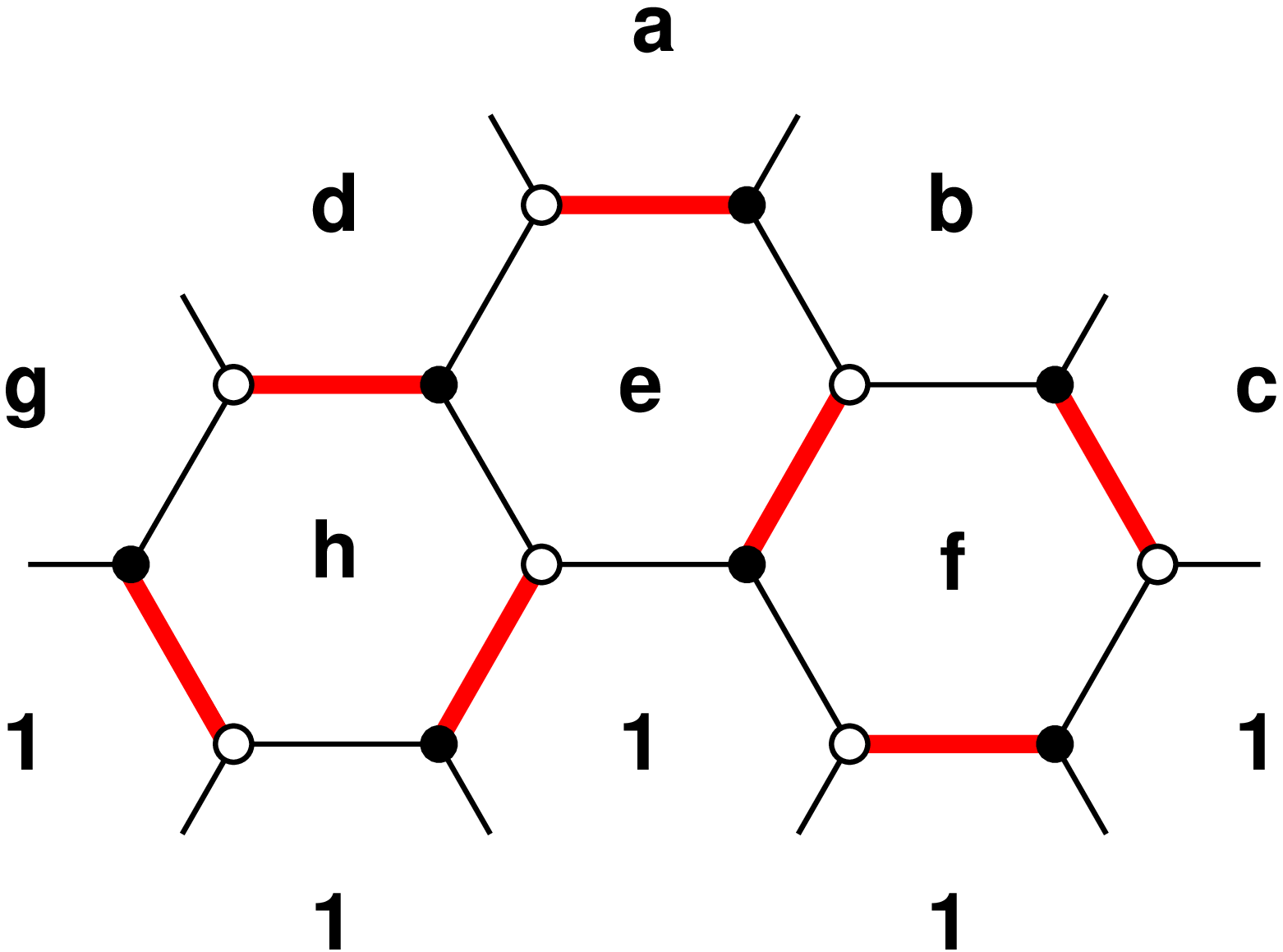}}} 
& \raisebox{-.3cm}{\hbox{\epsfxsize=2.5cm \epsfbox{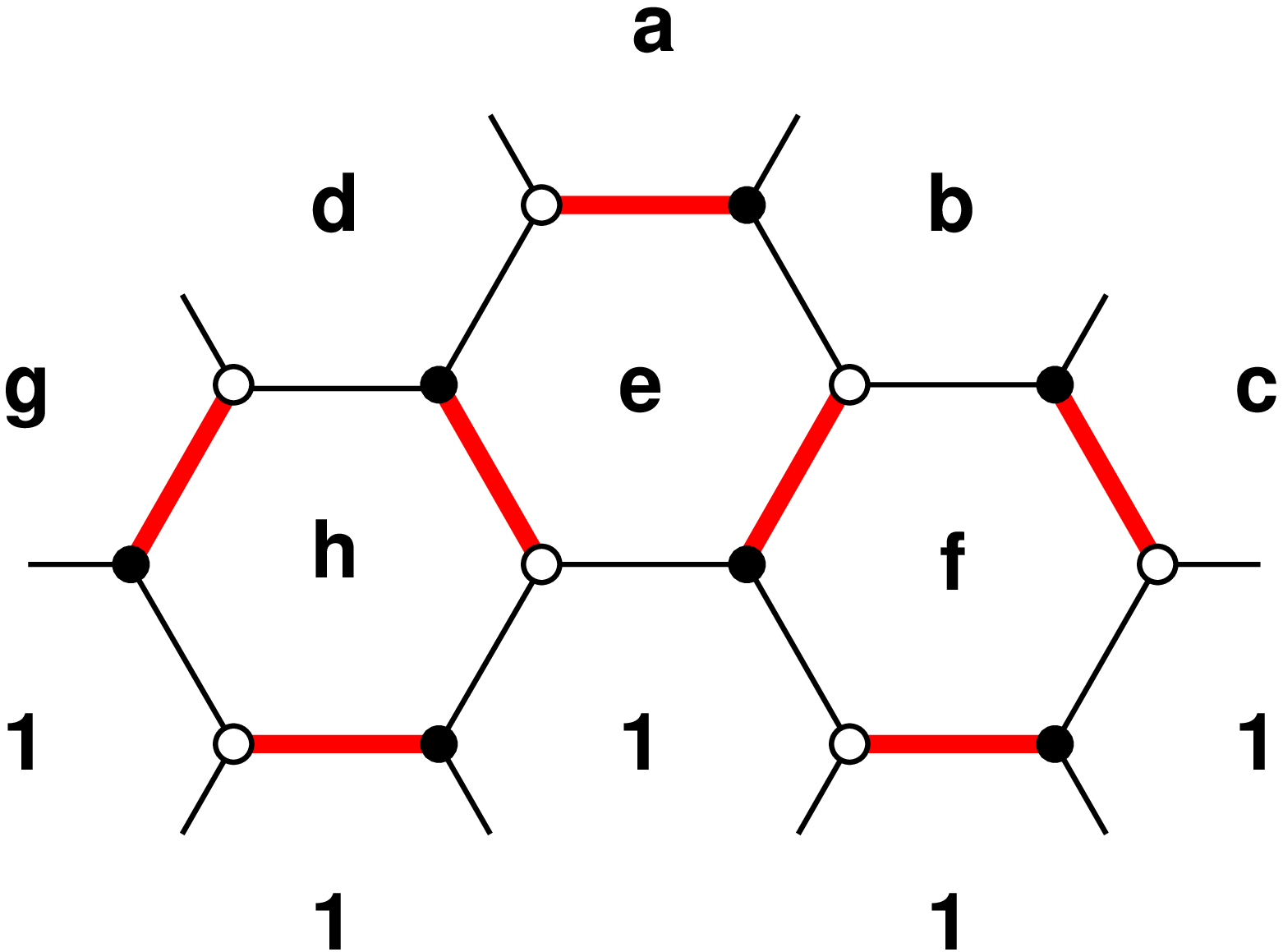}}}
&\raisebox{-.3cm}{\hbox{\epsfxsize=2.5cm \epsfbox{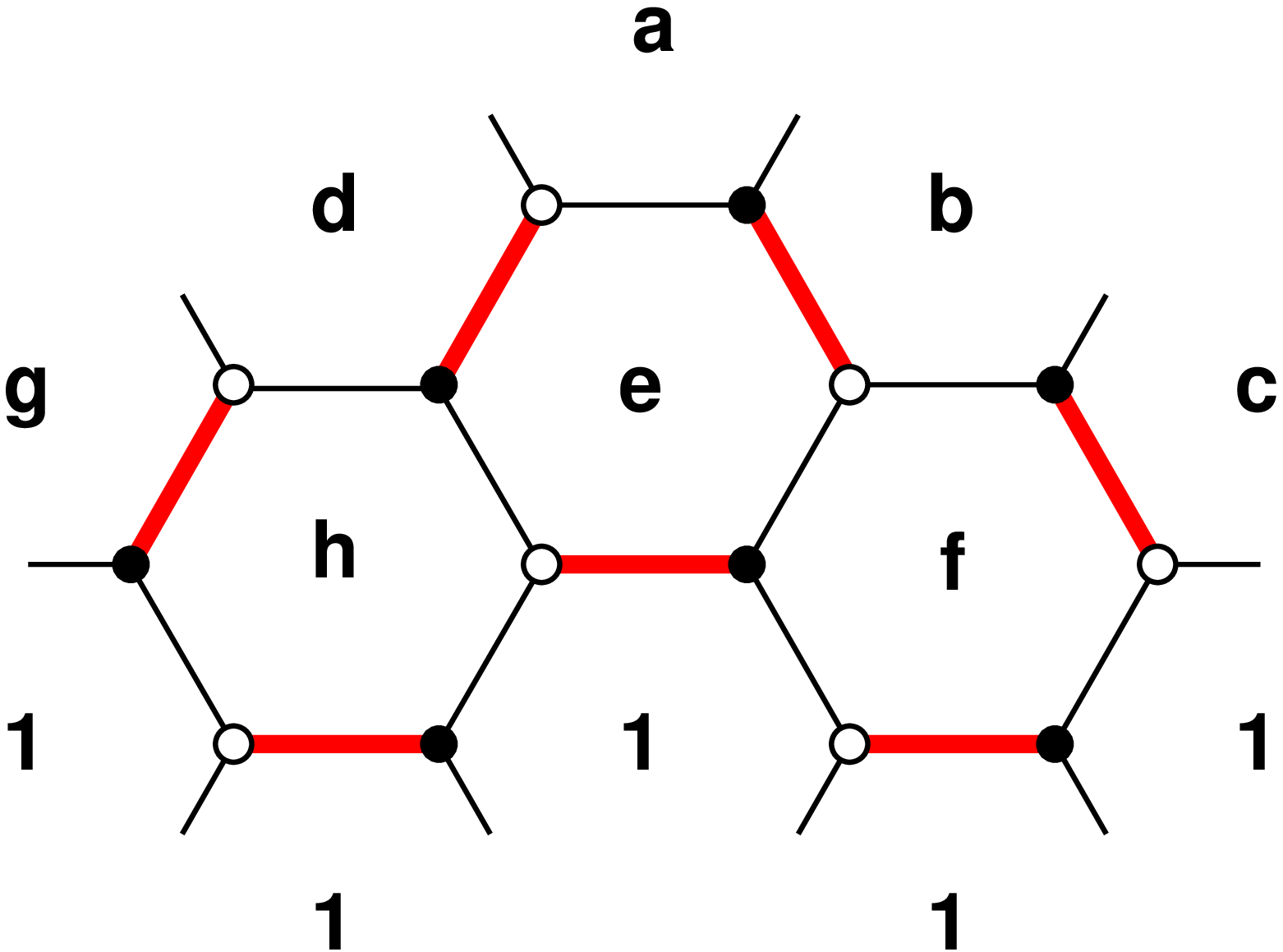}}} \\
{\rm weight}: &\frac{g e}{f h} &\frac{d c}{fh} &\frac{g b}{f h} &\frac{b d}{efh} & \frac{a}{e}\\
\end{matrix}
$$
\end{example}

We note that Theorems \ref{thpath} and \ref{thdim} may be connected more directly by showing that the path and dimer configurations
are in bijection with each-other. To best see this, recall that the dimer configurations on a domain of  the hexagonal lattice are in
bijection with rhombus tilings of the dual (triangular) lattice, by means of three types of rhombi obtained by gluing two adjacent triangles along 
their common edge. 
For the above example of $\lambda_{a,b}=(5,4,4,4,2)$, the domain ${\mathcal T}_{a,b}$ of the triangular lattice dual to ${\mathcal G}_{a,b}$ reads:
$$  \raisebox{-1.3cm}{\hbox{\epsfxsize=8.cm \epsfbox{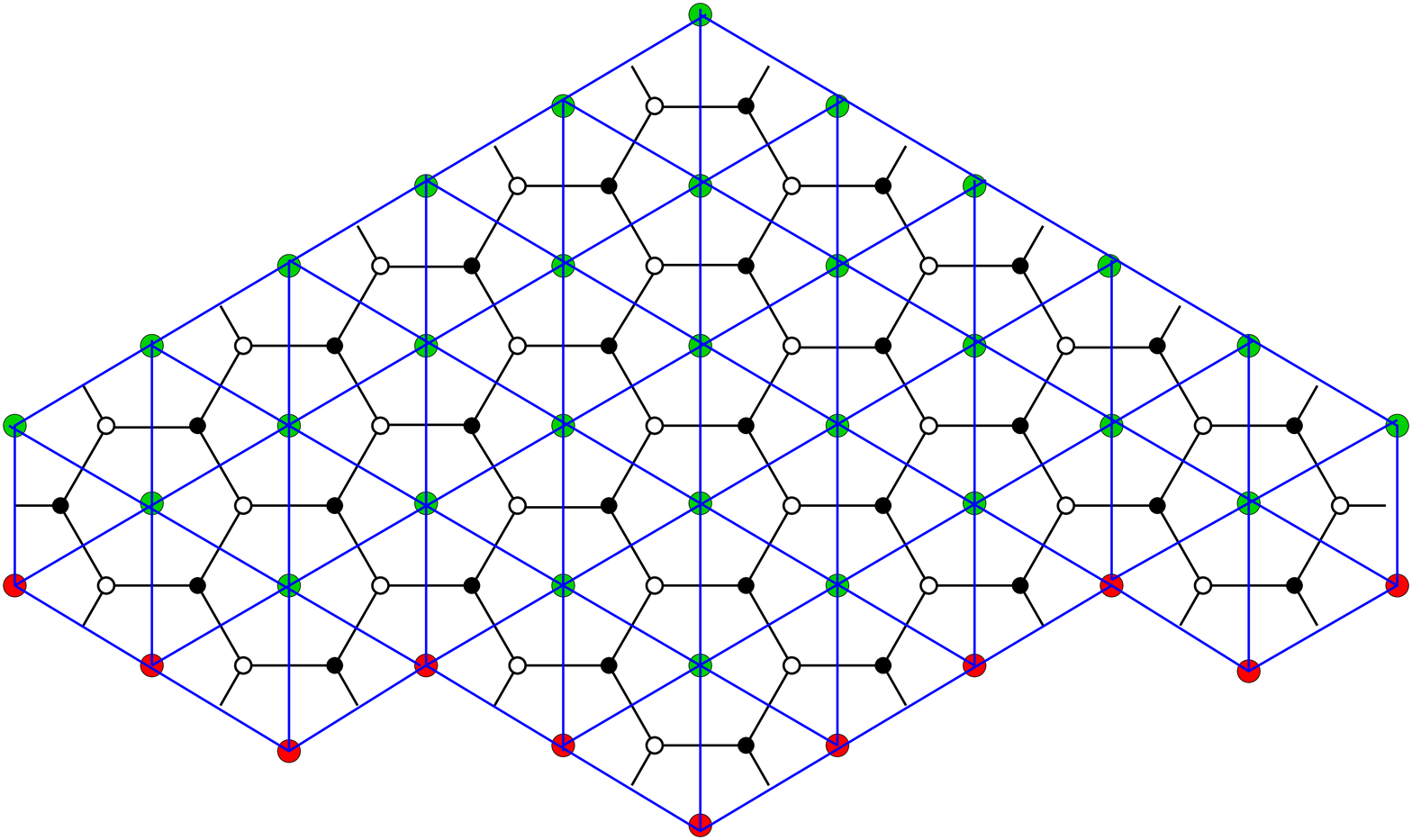}}} $$
Notice that generically the domain ${\mathcal T}_{a,b}$ only has two vertical boundary edges, dual to the only two horizontal external edges of 
${\mathcal G}_{a,b}$.
Dimer configurations on ${\mathcal G}_{a,b}$ are in bijection with rhombus tilings of ${\mathcal T}_{a,b}$. 
Moreover, such tilings are uniquely determined by either of three sets of non-intersecting paths of rhombi 
(the so-called De Bruijn lines) defined as follows. 
Each boundary edge of ${\mathcal T}_{a,b}$ has either of three orientations (vertical, $+30^\circ$, or $-30^\circ$). Starting from
any boundary edge, let us construct the chain of consecutive rhombi that share only edges of the same orientation. Such a chain is a path
connecting two opposite boundary edges. For a given orientation of the boundary edge, all such paths are non-intersecting, and form one of the above-mentioned three families. Any single such family determines the tiling entirely. In the present case, the ``vertical" family is particularly
simple, as it is made of a single path of rhombi:
$$  \raisebox{-1.3cm}{\hbox{\epsfxsize=8.cm \epsfbox{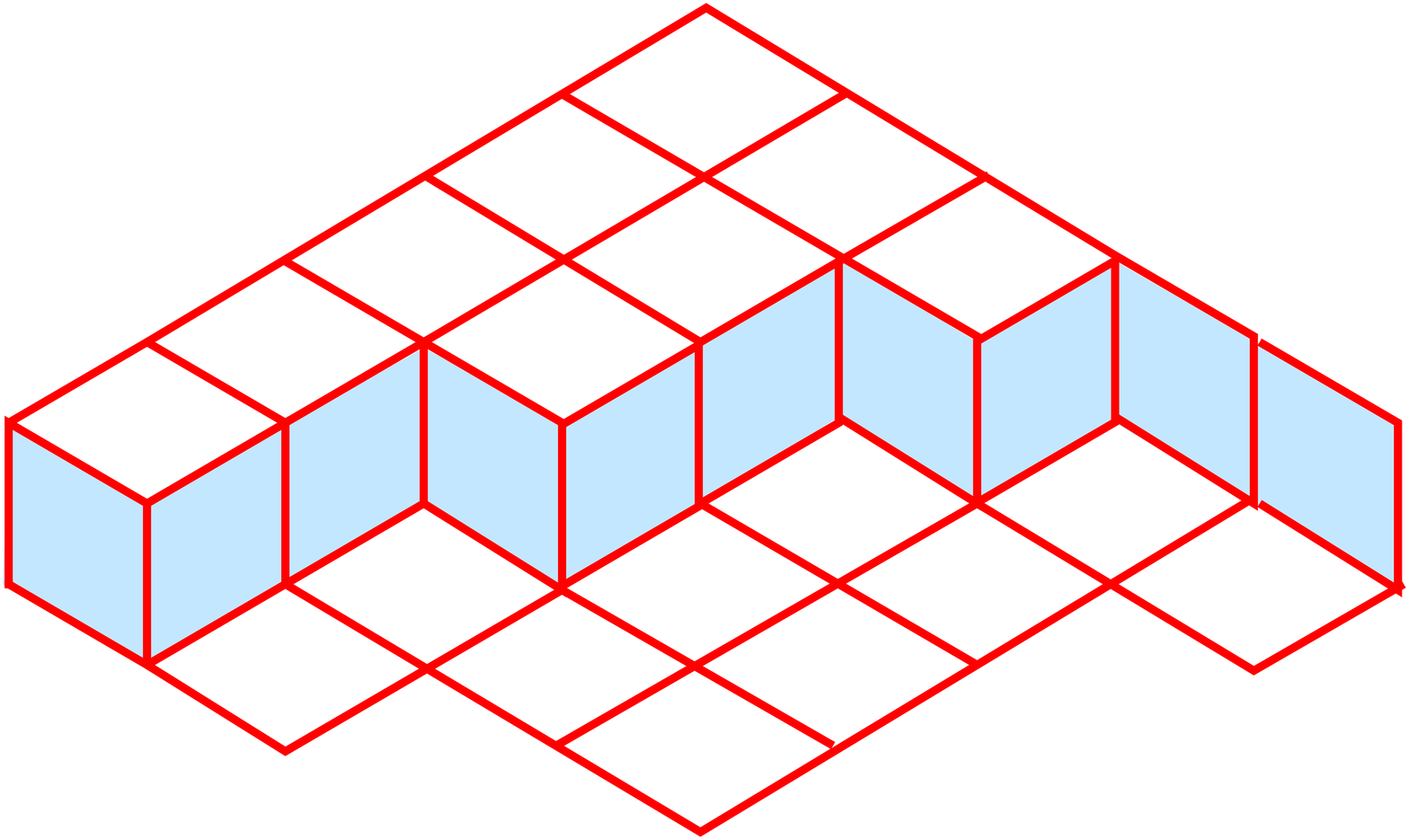}}} $$ 
This path determines the tiling entirely. We have therefore associated to each dimer configuration
on ${\mathcal G}_{a,b}$ a unique path with up and down steps (say from left to right), which can be represented on the network $L(j_0,j_1)={\mathcal N}_{a,b}$ as a path from the leftmost vertex to the rightmost one.
This gives a bijection between the dimer configurations of Theorem \ref{thdim} and the paths on networks of Theorem \ref{thpath}.
It is easy to show how the local weights of the path model can be redistributed into those of the dimer model, thus establishing directly the equivalence between the two theorems.

\subsection{Back to the polynomials of Bessenrodt and Stanley}\label{sectobs}

In this section we give two independent proofs of Theorem \ref{restrictoBS}.

As pointed out earlier, the polynomials of \cite{BS} may be recovered by specializing the variables $\theta_{i,j}$ according to \eqref{bsrestrict}.
More precisely, the (restricted) polynomials $p_{a,b}(x_\cdot)$ of \cite{BS} are defined by the formula:
\begin{equation}\label{BSp} p_{a,b}(x_\cdot)=\left\{ \begin{matrix}1 &\quad  {\rm if} (a,b)\in \lambda^*\setminus \lambda\\
{} & {}\\
\sum_{\mu \subset \lambda_{a,b}} \prod_{(r,s)\in \lambda_{a,b}\setminus \mu} x_{r,s} & {\rm if}\ (a,b)\in \lambda \end{matrix}\right. 
\end{equation}
where in the second formula the sum extends over all sub-diagrams $\mu$ of $\lambda_{a,b}$. In \cite{BS}, the determinant of 
the array $p_{a+i-1,b+j-1}(x_\cdot)$, $1\leq i,j\leq n_{a,b}$ pertaining to the square $S_{a,b}$ of $\lambda$ with NW corner $(a,b)$, was computed
to be equal to the ``leading term":
\begin{equation}\label{zdef} Z_{a,b}=\prod_{i=1}^{n_{a,b}} \prod_{(\al,\beta)\in \lambda_{a+i-1,b+i-1}} x_{\al,\beta}\ ,
\end{equation}
equal to the product of leading terms of 
$p_{a+i-1,b+i-1}(x_\cdot)$ along the first diagonal. With the choice of restrictions \eqref{bsrestrict} which identify $\theta_{a,b}$ with $Z_{a,b}$, 
and by uniqueness of the $T$-system solution,
we deduce that the polynomials $p_{a,b}(\theta_\cdot)$ defined by the $T$-system solution are identical to the polynomials 
$p_{a,b}(x_\cdot)$ of \cite{BS}, and Theorem \ref{restrictoBS} follows.

Let us now give an alternative direct proof of this result, by comparing the expression 
\eqref{BSp} to the restriction of the network expression of Theorem
\ref{thpath} for the solution of the $T$-system with steepest initial data stepped surface.
The first part of the formula \eqref{BSp}  is clear from the choice $\theta_{a,b}=1$ for  $(a,b)\in \lambda^*\setminus \lambda$. To 
recover the second part, first note that there is a bijection between the sub-diagrams $\mu \subset \lambda_{a,b}$ 
and the paths from the leftmost vertex
to the rightmost vertex on the network ${\mathcal N}_{a,b}$. To identify the polynomials,
we simply have to check that the weight of each path,
multiplied by the rightmost face variable ($t_{1,j_1}\equiv \theta_{a,\lambda_a^*}=1$ here) reduces to 
$\prod_{(r,s)\in \lambda_{a,b}\setminus \mu} x_{r,s}$, namely the product of $x$ variables under the path in $\lambda_{a,b}$.

To best compare the two settings, let us translate the network $L(j_0,j_1)$ globally by the vector by $(0,0,-2)$ in the $L_{FCC}$
representation (namely $k\to k-2$), so that the face labels match the positions of the corresponding boxes in $\lambda_{a,b}$.
This changes the local edge weight rules accordingly (by moving the face variables by two steps downwards).
In this new representation, the local edge weights of the path model can now be written in terms of the $x$ variables as:
$$  \raisebox{-1.3cm}{\hbox{\epsfxsize=14.cm \epsfbox{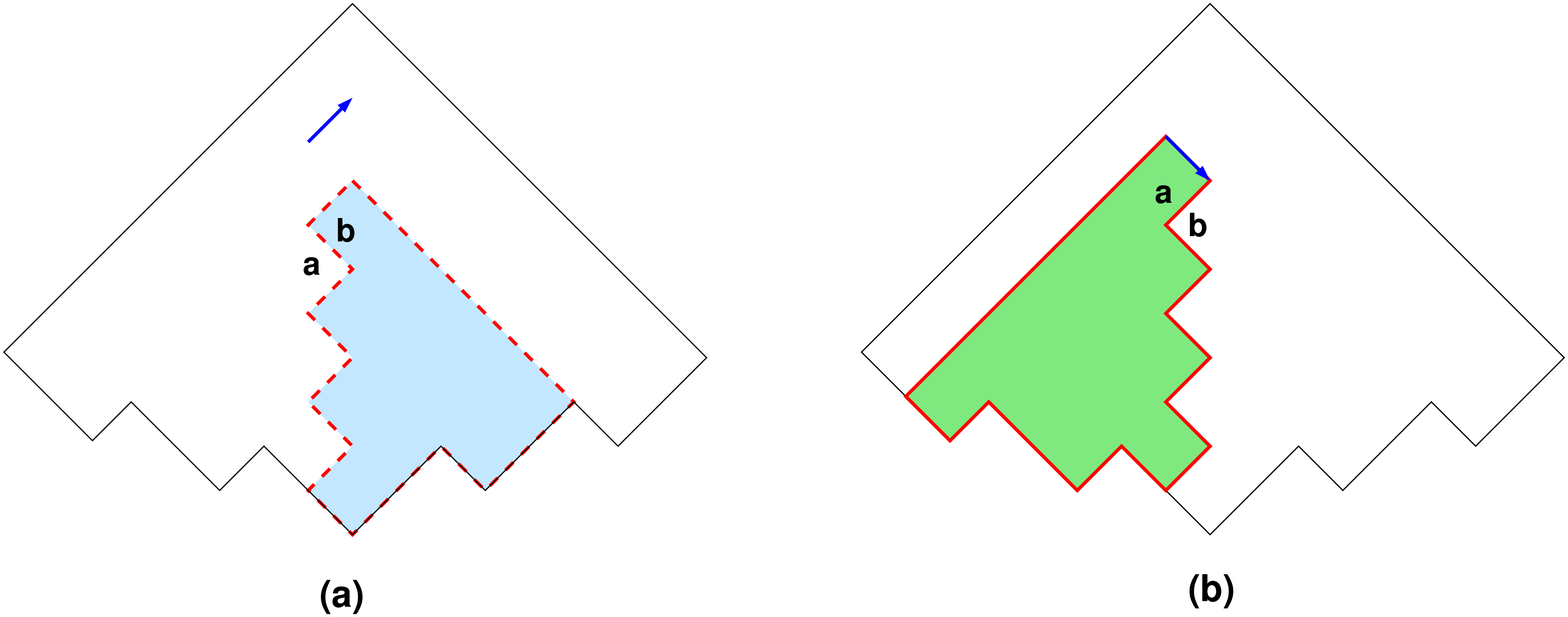}}}$$
where we have shown the two possible cases of an up or down-pointing edge of the path, both with weight $a/b$, and represented 
the weight in terms of the box variables $x$ as follows: in case (a) (up step), the weight $a/b$ is the inverse of the product of the $x$'s
in the dashed (blue) domain, while in case (b) (down step), the weight $a/b$ is the product of the $x$'s in the solid (green) domain.
We deduce that only boxes below the path delimiting $\mu$ in $\lambda_{a,b}$ contribute. Moreover, a given such box with variable $x$ receives a
contribution $x^{N_{\rm down}-N_{\rm up}}$, where $N_{\rm up}$ and $N_{\rm down}$ denote the total number of up/down steps of the path
that respectively belong to the left/right sector seen from the box as depicted below:
$$ \raisebox{-1.3cm}{\hbox{\epsfxsize=4.cm \epsfbox{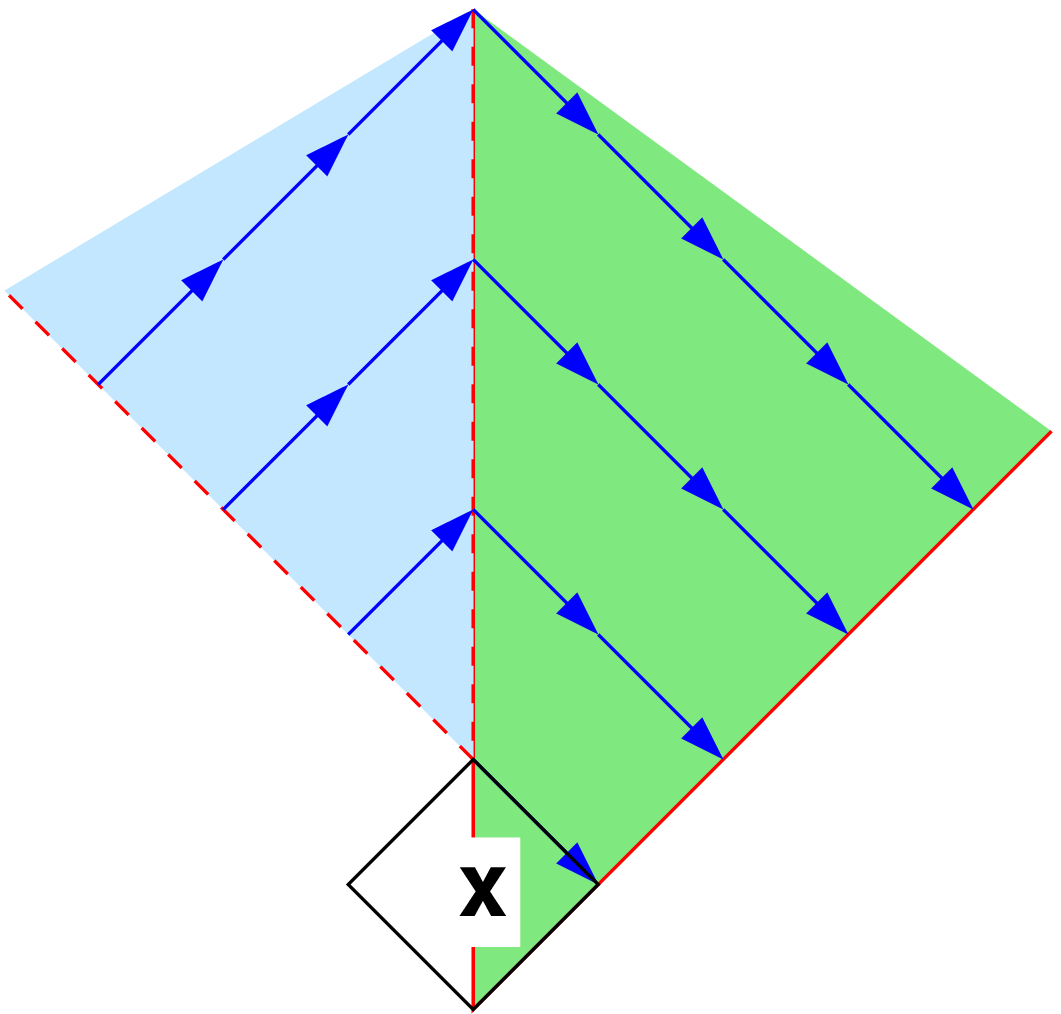}}} $$
It is clear that we always have $N_{\rm down}-N_{\rm up}=1$ as the path always goes up one step less in the left sector 
than it goes down in the right sector.
The total weight of the path delimiting $\mu$ in $\lambda_{a,b}$ is therefore the product over all the boxes below $\mu$ of the box variables $x$,
and the Theorem follows.

\section{3D Generalization}

\subsection{The general solution of the $T$-system}

So far we have concentrated on the solution $T_{j,k}=T_{1,j,k}$ of the $T$-system in the $i=1$ plane.
The general solution of the $T$-system gives access more generally to values of $T_{i,j,k}$ in other planes $i\geq 2$ as well.
In Ref.\cite{DF}, it was shown that such solutions are partition functions of families of $i$ non-intersecting paths on  the
same type of network as for $i=1$. More precisely, we first consider the base of the pyramid $\Pi_{i,j,k}$, which is a square
array of $T_{i+b-a,j+i+2-a-b}$, $a,b=1,2...,i$. We then construct all left and right projections of the points in this array,
say $\ell_1,...,\ell_i$ and $r_i,r_{i-1},...,r_1$ from left to right. Then the solution $T_{i,j,k}$ is given by the following:
\begin{thm}{\cite{DF}}
The solution $T_{i,j,k}$ of the $A_{\infty/2}$ $T$-system with initial data $(\bk,\bt)$ is equal to the partition function
of $i$ non-intersecting paths on the network $N(\ell_1,r_1)$ that start from the points $(\ell_a,k_{\ell_a})$ on the stepped surface $\bk$
and end at the points $(r_a,k_{r_a})$ on the stepped surface $\bk$, multiplied by the boundary term:
$ \prod_{a=2}^i t_{1,\ell_a}^{-1} \prod_{b=1}^i t_{1,r_a} $.
\end{thm}

\subsection{Pyramid of a partition}

Starting from a partition/Young diagram $\lambda=(\lambda_1,...,\lambda_N)$, we define its  {\it pyramid} $P_\lambda$ as the family
of partitions/Young diagrams $\lambda^{(i)}$, $i=1,2,...,k$, such that (i) $\lambda^{(1)}=\lambda$ (ii) $\lambda^{(i+1)}$ 
is obtained from $\lambda^{(i)}$ by removing its first row and column (iii) $\lambda^{(k)}=\emptyset$. 

The centers of boxes of the pyramid $P_\lambda$ may be represented as the set of vertices 
$(i,j,k)$ in $L_{FCC}$ such that the base of the pyramid $\Pi_{i,j,k}$ in the plane $i=1$ is entirely contained in $\lambda$.
In this correspondence, $\lambda^{(m)}$ is simply the intersection of the plane $i=m$ with this set of vertices. Another representation
of the pyramid $P_\lambda$ is as a strip decomposition of $\lambda$, by superimposing all the diagrams $\lambda^{(m)}$, with their $(1,1)$
box in the same position.
Finally, we define the extended pyramid $P_\lambda^*$ to be the pyramid of the extended Young diagram $\lambda^*$, namely
$P_\lambda^*=P_{\lambda^*}$.

\begin{example}
The pyramid of the partition $\lambda=(4,4,3)$ is  $\lambda^{(1)}=(4,4,3)$, $\lambda^{(2)}=(3,2)$, $\lambda^{(3)}=(1)$, 
$\lambda^{(4)}=\emptyset$. The representation of $P_\lambda$ in $L_{FCC}$ and as a strip decomposition of $\lambda$ are respectively:
$$ \raisebox{-1.5cm}{\hbox{\epsfxsize=6cm \epsfbox{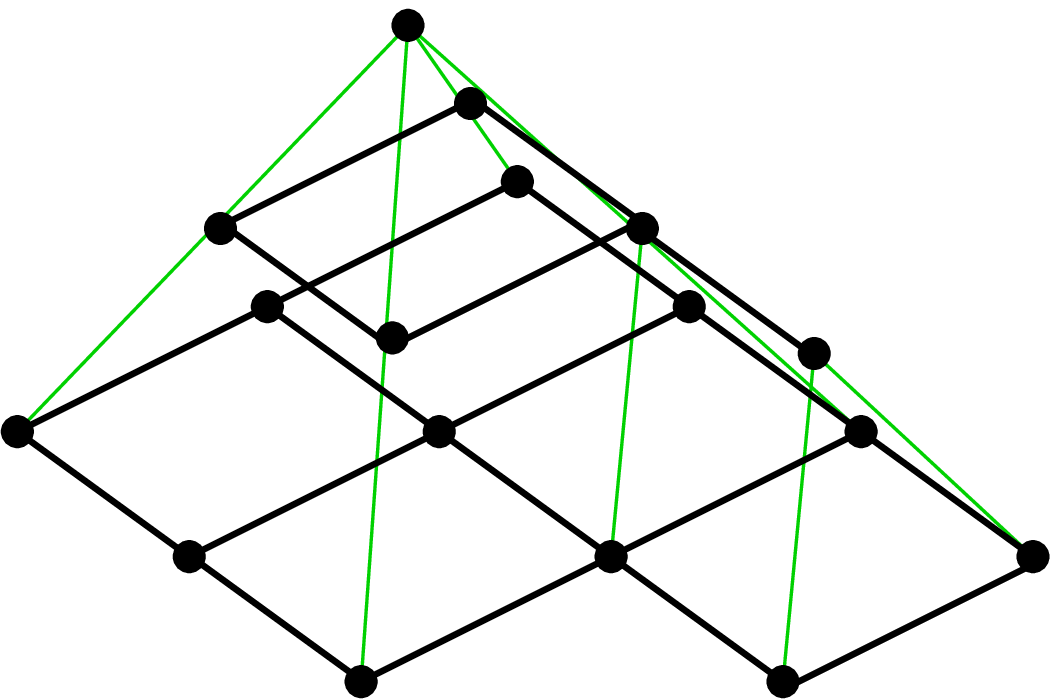}}} \quad {\rm and}\quad 
\raisebox{-1.5cm}{\hbox{\epsfxsize=4cm \epsfbox{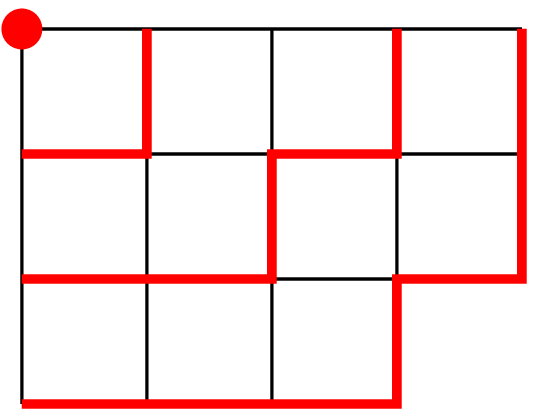}}}$$
\end{example}

\subsection{The family of Laurent polynomials for a pyramid}

The solution of the $A_{\infty/2}$ $T$-system within a pyramid $P_\lambda$ defined as above is entirely fixed by the 
assignment of initial data on the ``roof" of the pyramid, defined by $P_\lambda^*\setminus P_\lambda$. Note that this roof is nothing but the 
portion of the steepest stepped surface that determines the polynomials $p_{a,b}$ entirely. This leads to a natural extension of the family
of polynomials $\{p_{a,b}\}_{(a,b)\in \lambda}$ into a pyramid family $\{p_{a,b,m}\}$ with $1\leq m\leq k$ and
$(a,b)\in \lambda^{(m)}$, obtained by identification of the solution $T_{i,j,k}$ at the corresponding vertex of $P_\lambda\subset L_{FCC}$.

As a consequence of this definition, we may rewrite \eqref{wrons} as:

\begin{thm}
The pyramid polynomials $p_{a,b,m}(\theta_\cdot)$ associated to a Young diagram $\lambda$ are entirely determined by the determinant identity:
\begin{equation}\label{deterp} p_{a,b,m}(\theta_\cdot)=\det\left(\left( p_{a+i-1,b+j-1}(\theta_\cdot)\right)_{1\leq i,j,\leq m}\right) \end{equation}
where the array of points in the determinant corresponds to the base in the plane $i=1$ of the pyramid with apex at the center of the box 
$(a,b)$ of $\lambda^{(m)}$ in the $L_{FCC}$ representation.
\end{thm}

The network interpretation of the solution of the $T$-system \cite{DF} allows to immediately interpret the pyramid polynomial $p_{a,b,m}$
as the partition function for $m$ non-intersecting paths on the network ${\mathcal N}_{a,b}$ associated to $\lambda_{a,b}$ 
and the steepest stepped surface, up to a multiplicative boundary factor, by direct application of the Lindstr\"om Gessel-Viennot Theorem
\cite{LGV1,LGV2}. These paths start/end at the points of 
$\bk$ that correspond to the left/right projections of the array of points in the determinant \eqref{deterp}. On ${\mathcal N}_{a,b}$,
these are the left/right  projections of the top vertex of each box in the corresponding square of size $m$ with top box $(a,b)$.

\begin{example}
Let us consider $\lambda=(5,4,4,4,3)$ and the network $L(j_0,j_1)={\mathcal N}_{1,1}$ of Sect.\ref{networksec}. 
The pyramid polynomial $p_{1,1,3}$
is equal to the determinant $\tiny \left\vert \begin{matrix} p_{1,1}& p_{1,2} & p_{1,3} \\
p_{2,1} & p_{2,2} & p_{2,3} \\
p_{3,1} & p_{3,2} & p_{3,3} \end{matrix}\right\vert$. 
It is also proportional to the partition function of paths from the entry to the exit vertices marked
$1,2,3$ below:
$$ \raisebox{0.cm}{\hbox{\epsfxsize=16.cm \epsfbox{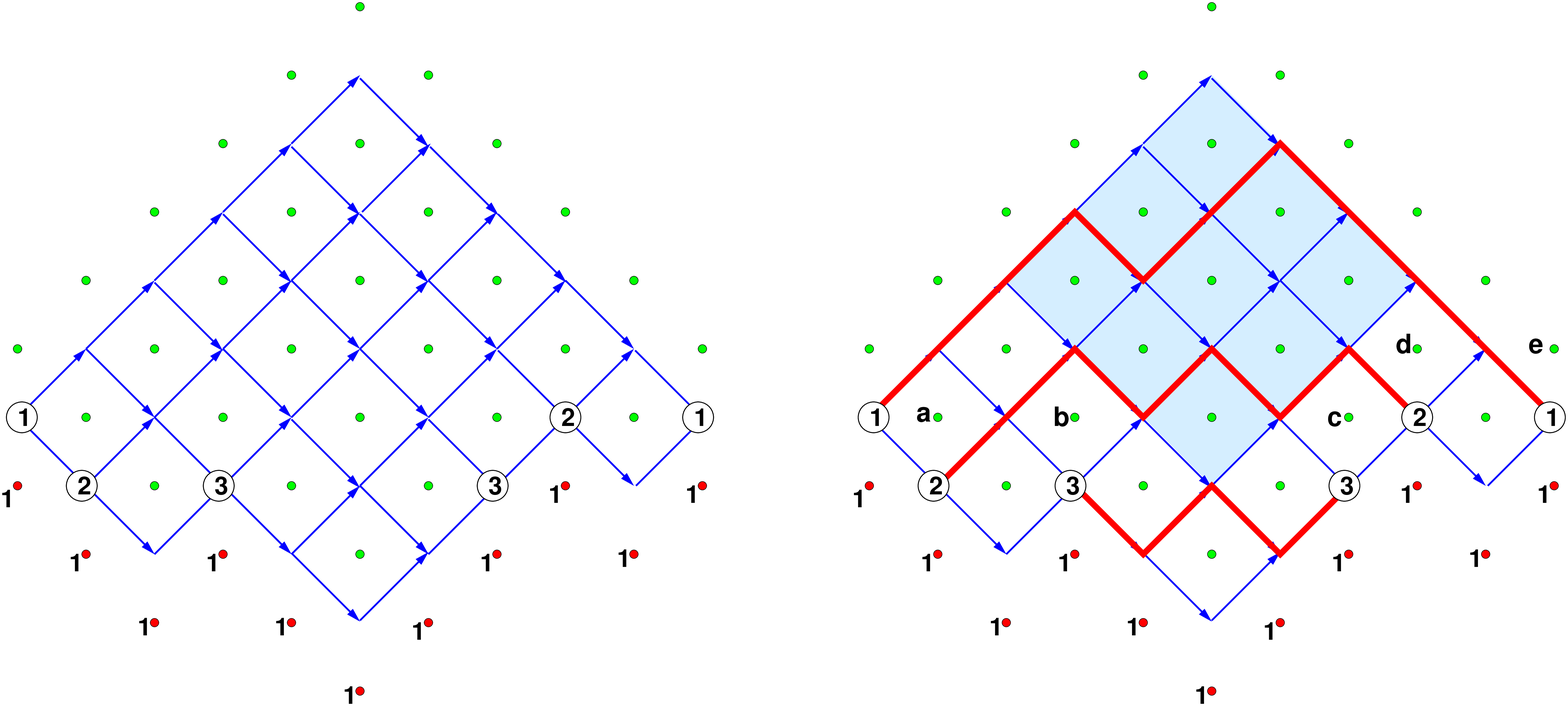}}}  $$
where we have also represented a sample configuration of these three non-intersecting paths. 
The entry/exit vertices of ${\mathcal N}_{a,b}$ are the left/right projections of the vertices at the top of the boxes in the corresponding square
(here shaded in blue).

The proportionality factor
is simply $a^{-1}b^{-1} c d e$, where the corresponding face variables are immediately above the entry points $2,3$ and exit points $1,2,3$.
\end{example}

\subsection{Generalized Bessenrodt-Stanley polynomials for a pyramid}

We may now restrict the pyramid Laurent polynomials $p_{a,b,m}(\theta_\cdot)$ attached to a partition $\lambda$ via the same
change of variables \eqref{bsrestrict} to box weights $x_{i,j}$. This leads us to the definition of the quantities:
\begin{equation}\label{BSpm}p_{a,b,m}(x_\cdot)=\det\left(\left( p_{a+i-1,b+j-1}(x_\cdot)\right)_{1\leq i,j,\leq m}\right) \end{equation}


We have:

\begin{thm}\label{genBS}
The pyramid polynomials $p_{a,b,m}(x_\cdot)$ \eqref{BSpm} have non-negative integer coefficients, and moreover
$p_{a,b,m}(x_\cdot)$ is the partition function for $m$ non-intersecting paths on the network ${\mathcal N}_{a,b}$ associated to $\lambda_{a,b}$,
from the $m$ left projections to the $m$ right projections of the vertices of the square of size $m$ with NW corner at $(a,b)$.
Alternatively, $p_{a,b,m}(x_\cdot)$ is the partition function for $m$ strictly nested partitions 
$\emptyset\subset \mu_m\subset\mu_{m-1}\subset\cdots\subset\mu_1\subset\lambda_{a,b}$ with $\mu_i\subset\lambda_{a,b}^{(i)}$
inside $\lambda_{a,b}$, with the usual weights:
$$p_{a,b,m}(x_\cdot)=\sum_{\mu_i\subset\lambda_{a,b}^{(i)}\atop i=1,2,...,m}
\prod_{i=1}^m \prod_{(\al,\beta)\in \lambda_{a,b}\setminus\mu_i}x_{\al,\beta}$$
\end{thm}

\section{Conclusion/Discussion}

In this paper, we have expressed the polynomials of \cite{BS} as particular solutions of the octahedron recurrence 
in a half-space with ``steepest" initial data surface attached to a fixed partition $\lambda$. This connection has allowed us to generalize these polynomials to the full pyramid $P_\lambda$, and to find alternative expressions involving paths on networks or dimers on bipartite graphs.

More generally, we may consider different initial data surfaces associated to the partition $\lambda$, and use the solutions of the corresponding 
$T$-system to define different classes of polynomials attached to the boxes of $\lambda$. Another natural choice is to pick the stepped surface to be made of ``vertical walls" along the boundary of $\lambda$, namely with vertices alternating between two parallel planes with normal vector
$(0,1,1)$ or $(0,-1,1)$ in $L_{FCC}$. In the case of $\lambda=(5,5,5,5,3)$, the 3D FCC lattice view and the corresponding 
stepped surface lozenge decomposition look like:
$$  \raisebox{-1.3cm}{\hbox{\epsfxsize=8.cm \epsfbox{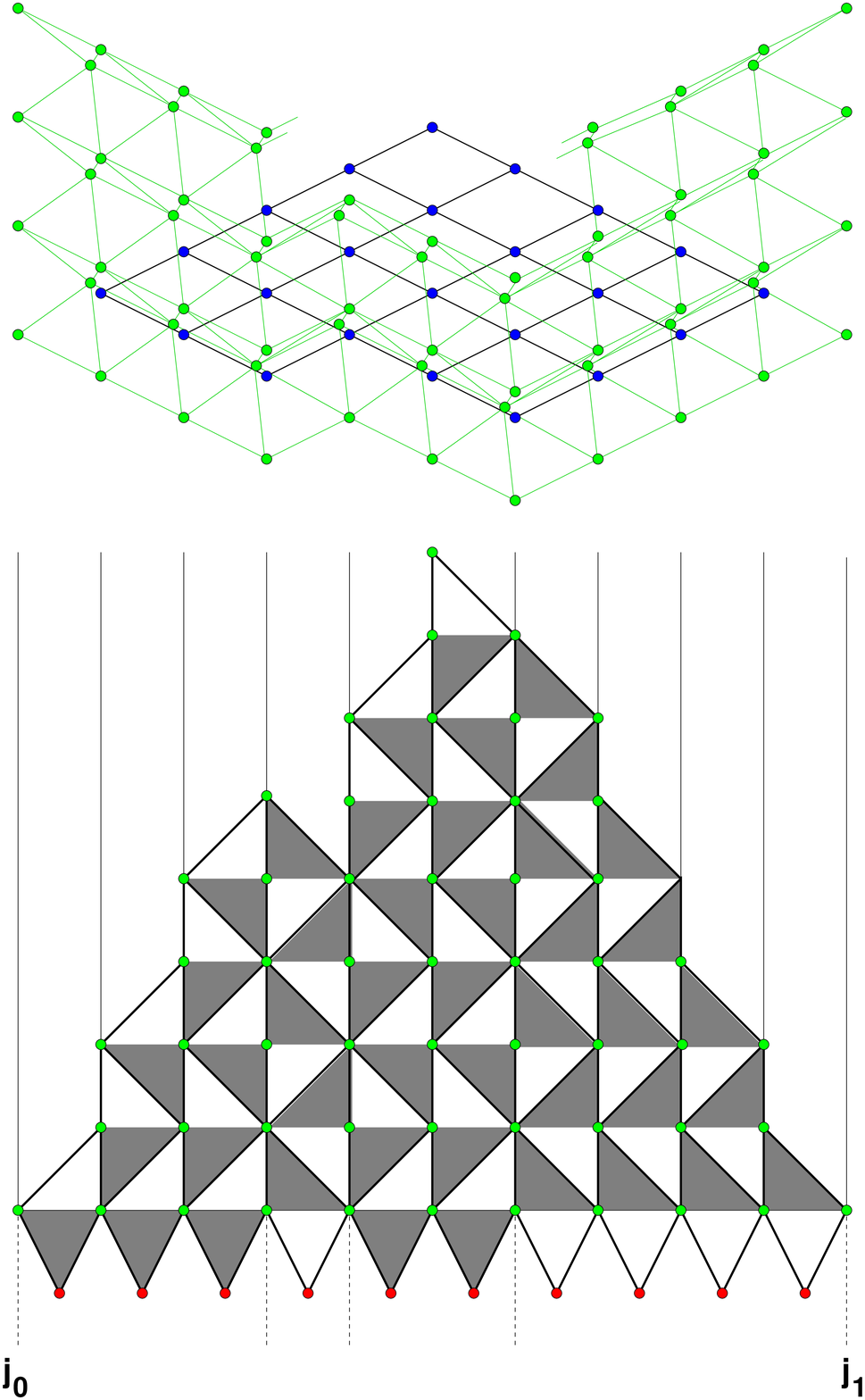}}} $$ 
Here the vertical wall stepped surface is represented in front, with green vertices. 
The centers of the boxes of $\lambda$ are the blue vertices in the bottom plane ($i=1$).
In the lozenge decomposition of the stepped surface, we have only represented the lozenges that will contribute to the 
solution within $\lambda$, and added as before triangles in the bottom row, with their bottom-most vertex assigned value $1$
(the $A_{\infty/2}$ boundary condition \eqref{plus}).
Note also that the boundary vertical walls intersect the plane $i=1$ along the boxes of $\lambda^*\setminus \lambda$.

The difference with the situation of the steepest stepped surface is subtle: in both cases, the walls are made uniquely of either $U$ type of $V$ type lozenges, but their arrangement (the order in which they come from left to right) is different, due to the rules \eqref{eqrules}.
We may now translate the lozenge decomposition into a network, which we deform in the same way as before to finally get:
$$   \raisebox{-1.3cm}{\hbox{\epsfxsize=16.cm \epsfbox{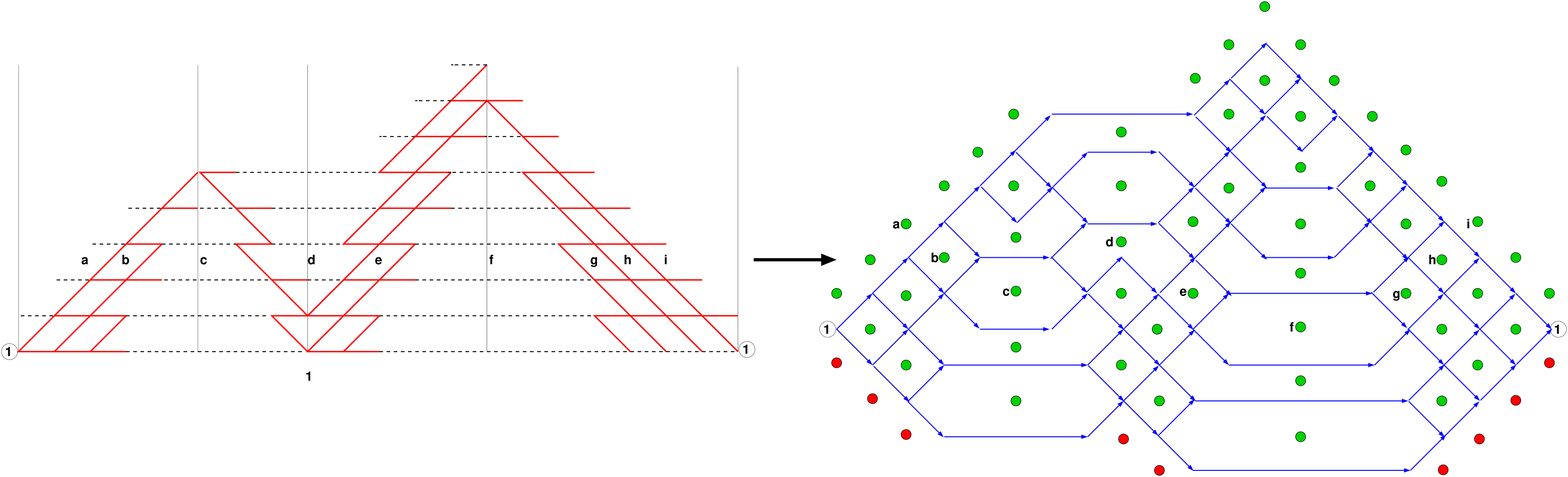}}}  $$
where we have indicated the correspondence between a sample row of face weights (green dots  in general)
with assigned values $t_{i,j}=a,b,c,...$, and the bottom outer faces by red dots (with assigned values $1$). 
The edge weights in the network are 
related to the face variables in the usual way \eqref{rulnet}, and all the horizontal edges receive the weight $1$. 
The construction of the network looks more complicated, but again the solution at the box $(1,1)$ of $\lambda$ is the partition function for paths on this network, from the leftmost to the rightmost vertex. In fact, it is possible to deform the network to make it match the shape of the Young 
diagram $\lambda$, at the expense of adding some extra edges as follows:
$$  \raisebox{-1.3cm}{\hbox{\epsfxsize=9.cm \epsfbox{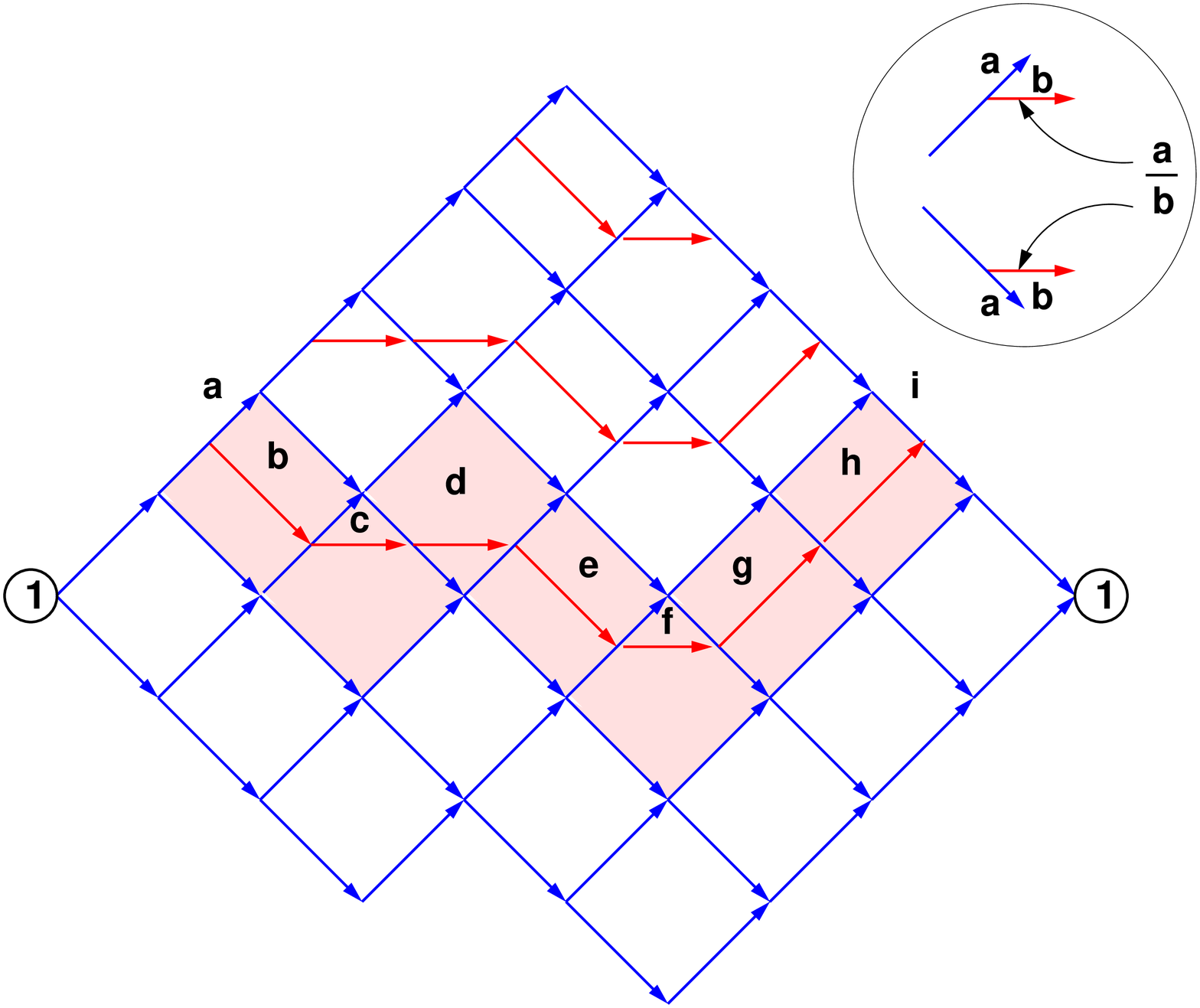}}}$$
where the up/down edge weights are related to the face variables in the usual way, while the new horizontal edges have weights as indicated in the medallion. The additional arrows (in red) split up each box of the strip decomposition $\{ \lambda^{(i)}\setminus\lambda^{(i+1)}\}_{i\in [2,k-1]}$ of 
$\lambda$, where $\lambda^{(i)}$ is the $i$-th layer of the pyramid $P_\lambda$,  by connecting centers of box edges (we have shaded the strip
$\lambda^{(2)}\setminus\lambda^{(3)}$ in the above example). In particular all paths of red arrows are parallel. Note also that the first strip
$\lambda^{(1)}\setminus\lambda^{(2)}$ is not split.
Let us denote by ${\tilde {\mathcal N}}_{a,b}$
the network thus constructed out of the partition $\lambda_{a,b}$.

Having assigned fixed parameters $t_{i,j}$ to each
inner and outer face of the resulting network, we may associate a Laurent polynomial $q_{a,b,m}(t_\cdot)$ to the $(a,b)$ box of the $m$-th layer 
$\lambda^{(m)}$ of the pyramid of $\lambda$, equal to the corresponding $T$-system solution with vertical wall 
initial data stepped surface. Then we have:

\begin{thm}
The Laurent polynomial $q_{a,b,m}(t_\cdot)$ is the partition function of $m$ non-intersecting paths on ${\tilde {\mathcal N}}_{a,b}$, starting/ending
at the left/right projections of the top vertices of the boxes in the square with top box $(a,b)$ and size $m$, multiplied by the factor
$\prod_{\alpha=2}^m t_{1,\ell_\al}^{-1}\prod_{\beta=1}^m t_{1,r_\beta}$, where $\ell_\alpha$/$r_\alpha$ is the $j$ coordinate in the FCC lattice of the $\alpha$-th left/right projection of the points in the square of size $m$ with NW corner at $(a,b)$.
\end{thm}

\begin{example}
Let us consider the case $\lambda=(2,2)$. Let us assign the following initial values on the vertical walls:
$$\begin{matrix}
 & & T_{4,0,1}=\ell & & \\
& T_{3,-1,1}=i & T_{3,0,0}=j & T_{3,1,1}=k & \\
 & T_{2,-1,2}=f & T_{2,0,1}=g & T_{2,1,2}=h & \\
T_{1, -2, 2}= a & T_{1, -1, 1}= b & T_{1, 0, 0}= c & T_{1,1,1}=d & T_{1,2,2}=e 
\end{matrix}$$
Then the network ${\tilde {\mathcal N}}_{1,1}$ reads, with the face variables or alternatively the edge weights:
$$  \raisebox{-1.3cm}{\hbox{\epsfxsize=10.cm \epsfbox{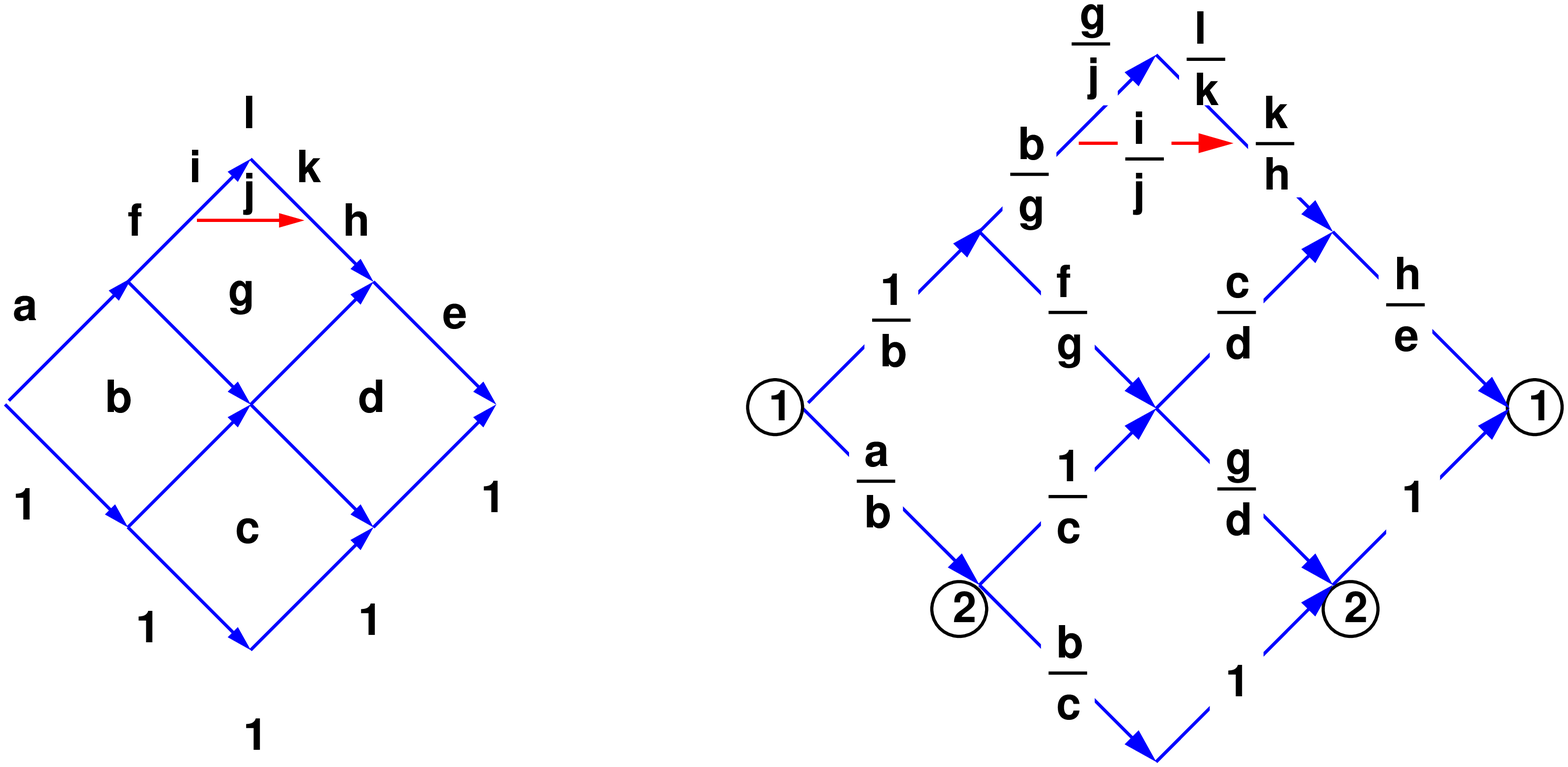}}}$$
Let us first compute $q_{a,b}=q_{a,b,1}$ for the boxes for the Young diagram $\lambda$.
The quantities $q_{1,1}/e,q_{1,2}/e,q_{2,1}/d,q_{2,2}/d$ are the partition functions for paths on ${\tilde {\mathcal N}}_{1,1}$,
respectively from vertices $1\to 1$, $2\to 1$, $1\to 2$ and $2\to 2$. This gives:
$$ \begin{matrix} q_{1,1}=
\frac{a e}{c} + \frac{e f}{b d} + \frac{a e g}{b c d} + \frac{a h}{b d} + \frac{c f h}{
 b d g} + \frac{t}{y} + \frac{x z}{g y} & q_{1,2}=\frac{b e}{c} + \frac{e g}{c d} + \frac{h}{d} \\
 q_{2,1}= \frac{a d}{c} +  \frac{f}{b} +  \frac{a g}{b c} & q_{2,2}=\frac{b d}{c} + \frac{g}{c}
 \end{matrix} $$
 Finally the quantity $q_{1,1,2} b/(d e)$ is the partition function for pairs of non-intersecting paths from $1,2$ to $1,2$:
 $$ q_{1,1,2} = \frac{f h}{g} + \frac{b d t}{c y} + \frac{g t}{c y} +\frac{x z}{c y}+\frac{b d x z}{ c g y} $$
\end{example}

%


\end{document}